\documentclass[leqno,12pt]{article}
\usepackage[intlimits]{amsmath}
\usepackage{amsmath,amsxtra,amssymb,latexsym, amscd,amsthm,pb-diagram,fancyhdr}
\usepackage[mathscr]{eucal}
\usepackage{indentfirst} 
\def\leq{\leqslant}
\def\geq{\geqslant}
\def\le{\leqslant}
\def\ge{\geqslant}
\usepackage{amssymb}
\usepackage{array}
\usepackage{hhline}
\usepackage{longtable}
\usepackage{tabularx}

\usepackage[pctex32]{graphics}
\usepackage{amsmath}
\usepackage[mathscr]{eucal}
\usepackage{amsfonts}
\usepackage{amsthm}
\usepackage{longtable}

\newtheorem{thm}{Theorem}[section]
\newtheorem*{thm*}{Theorem}
\newtheorem{cor}[thm]{Corollary}
\newtheorem{lem}[thm]{Lemma}

\theoremstyle{definition}
\newtheorem{defn}[thm]{Definition}
\theoremstyle{remark}
\newtheorem{rem}[thm]{Remark}

\def\G{{\mathscr G}}

\def\Eb{{\mathbb E}}

\def\Nb{{\mathbb N}}

\def\Pb{{\mathbb P}}

\def\Rb{{\mathbb R}}
\def\Sb{{\mathbb S}}

\def\Zb{{\mathbb Z}}

\def\B {{\mathcal B}} 

\def\D {{\mathcal D}}
\def\E {{\mathcal E}}
\def\F {{\mathcal F}}
\def\G {{\mathcal G}}

\def\al{\alpha}
\def\be{\beta}
\def\de{\delta}

\def\ka{\kappa}

\def\Om{\Omega}
\def\om{\omega}

\def\ep {\varepsilon}
\def\phi{\varphi}
\def\si{\sigma}

\def\th{\theta}

\def\et{\eta}
\def\ze{\zeta}

\def\bar{\overline}
\def\lam{\lambda}

\usepackage{geometry}
\renewcommand{\proofname}{Proof}

\def\bar{\overline}

\def\leq{\leqslant}
\def\geq{\geqslant}

\begin{document}
\title{Some results on regularity and monotonicity of the speed for excited 
 random walks in low dimensions}         
\author {Cong-Dan Pham\\Duy Tan University, Da Nang, Viet Nam\\ Aix-Marseille University, Marseille, France\\cong-dan.pham@univ-amu.fr}
\maketitle
\begin{abstract}
	Using renewal times and Girsanov's transform, we prove that the speed of the excited random walk is infinitely differentiable with respect to the bias parameter in $(0,1)$ for the dimension $d\ge 2$. At the critical point $0$, using a special method, we also prove that the speed is differentiable and the derivative is positive for every dimension $2\leq d\neq 3.$ However, this is not enough to imply that the speed is increasing in a neighborhood of $0.$ It still remains to prove the derivative is continuous at $0$.
 Moreover, this paper gives some results of monotonicity for $m-$excited random walk when $m$ is large enough or $m=+\infty.$ 
 \end{abstract}
\section{Introduction}
\subsection{Excited random walk}
 Excited random walk (ERW) was introduced in 2003 by I. Benjamini and D.B. Wilson \cite{BW03}. After that, M. Zerner in 2005 studied a generalization called multi-excited random walk on integers. The more general version of ERW called excited random walk in random environment was considered in \cite{MPRV12}, \cite{KZ14}. In this paper, we consider $m-$excited random walk ($m-$ERW) that we define as follows.
  
Let us describe the model: an $m-$ERW with bias parameter $\be\in[0,1]$ is a discrete time nearest neighbor random walk $(Y_n)_{n\ge 0}$ on the lattice $\Zb^d$ obeying the following rule: 
When the walk visits a site for the $k$-th time 
\begin{itemize}
\item if $k>m$, it jumps uniformly at random to  one of the $2d$ neighboring sites,
\item if $k\leq m$, it jumps with probability $(1+\be)/2d$ to the right, probability $(1-\be)/2d$ to the left, and probability 1/(2d) to the other nearest neighbor sites.
\end{itemize}
Let $(e_i:1\leq i\leq d)$ denote the canonical generators of the group $\Zb^{d}.$ From the description above of $m-$ERW, the law $\Pb_{m,\be}$ of $m-$ERW which is the probability on the path space $(\Zb^d)^\Nb$, is defined by: 
\begin{itemize}
\item $\Pb_{m,\be}(Y_0=0)=1$,
\item $\Pb_{m,\be}[Y_{n+1}-Y_n=\pm e_i|Y_0, \cdots, Y_n]=\frac{1}{2d}$ for $2\leq i\leq d$,
\item Denote by $\{Y_n\notin_k\}$ the event that $Y_n$ has been visited  exactly $k$ times at time $n$, i.e. $\{Y_n\notin_k\}:=[\#\{i\le n:Y_i=Y_n\}=k]$. For $k=1$ we also denote $\{Y_n\notin\}:=\{Y_n\notin_1\}$. Denote $\{Y\in_k\}$ and $\{Y_n\in\}$ be respectively the compliment events of $\{Y_n\notin_k\}$ and $\{Y_n\notin\}$. Let $\F_n$ be the $\si$-algebra $\si(Y_0,Y_1,...,Y_n)$ generated by the random walk up to time $n$. By the definition of $m$-ERW, we have
\begin{equation}\label{defE}
\Pb_{m,\be}[Y_{n+1}-Y_n=\pm e_1|\F_n,Y_n\notin_k]=
 \frac{1\pm \beta 1_{k\le m}}{2d}.
\end{equation}
\end{itemize}

When $m=1$ we recover excited random walk. Denote by $\Pb_{\be},\Pb_{m,\be}$ the laws (resp. $\Eb_{\be},\Eb_{m,\be}$ the expectations) respectively of $1$-ERW, $m$-ERW. Let $N_n$ be the numbers of visited points by the random walk $Y$ at time $n$ i.e the range of the random walk $Y$ up to time $n$ then $N_n=\sum_{i=0}^{n-1}1_{Y_i\notin}$. When $m=+\infty$ we obtain simple random walk with bias $\be$ (SRW). Set $R_n(\be):=\frac{\Eb_{\infty,\be}(N_n)}{n}$. It has been proved that (see \cite{Spi01}, \cite{LGR91}), $\Pb_{\infty,\be}$-a.s. the limit $lim_{n\to +\infty}\frac{N_n}{n}$ exists. It implies the existence of $lim_{n\to +\infty}\frac{R_n(\be)}{n}$ and denote this limit by $R(\be).$

In 2007, J. B$\acute{\text{e}}$rard and A. Ram$\acute{\text{\i}}$rez \cite{BR07} proved a law of large numbers and a $1$-dimensional result of central limit theorem for the $1$-ERW with $d\ge 2$ and $\be>0$, namely:
\begin{itemize}
\item (Law of large numbers). There exists $\ v=v(\be,d), 0<v<+\infty$ such that a.s.
$$\lim_{n\to\infty}n^{-1}Y_n\cdot e_1=v. $$
\item  (Central limit theorem). There exists $\si=\si(\be,d),0<\si<+\infty$ such that
$(n^{-1/2}(Y_{\lfloor nt\rfloor}\cdot e_1-v\lfloor nt \rfloor), t \ge 0)$ 
converges in law as $n\to+\infty$ to a Brownian motion with variance $\si^2$.
\end{itemize}

In 2012, with the different approach, M. Menshikov, S. Popov, A. Ram$\acute{\text{\i}}$rez  and M. Vachkovskaia proved a law of large numbers and a $d$-dimentional result of central limit theorem for the excited random walk in random environment, see Theorem 1.2 of \cite{MPRV12}.

Before stating our result, we define some notions. Let $Y=(Y_n)_{n\geq 0}$ be the $1$-ERW. Let $\tau_1$ is the first renewal time of $1$-ERW (For more detail, see in the Section \ref{rene}, Definition \ref{tau} and Lemma \ref{esstau}). Informally, a renewal time (in direction $e_1$) is a maximum of the first component of the random walk $(Y_n\cdot e_1)_{n\geq 0}$ which is also a minimum of the future of the first component of the random walk. A time $n$ is a renewal time if 
$$
\sup_{0\leq i<n}Y_i\cdot e_1<Y_n\cdot e_1\leq \inf_{n\leq i } Y_i\cdot e_1.
$$
So, the first renewal time $\tau_1$ is defined by
$$\tau_1=\inf\{n\geq 0: \sup_{0\leq i<n}Y_i\cdot e_1<Y_n\cdot e_1\leq \inf_{n\leq i } Y_i\cdot e_1\}.$$ We also denote $\tau:=\tau_1.$ Set $\bar{D}:=\inf\{m\geq 0: Y_m\cdot e_1<Y_0\cdot e_1\}.$ For $1$-ERW, it has been proved in \cite{BR07} that $\Pb_{\be}(\bar{D})>0$. Therefore, we can define the conditional probability $\hat{\Pb}_{\be}(\cdot)=\Pb_{\be}(\cdot|\bar{D}).$ Let $\hat{\Eb}_{\be}$ be the expectation with respect to $\hat{\Pb}_{\be}.$

 Our main result about regularity for the $1$-ERW is the following:
\begin{thm}\label{smallD} For $d\geq 2$, let $v_n(\be)$ be the speed at time $n$ i.e. $v_n(\be)=\frac{\Eb_{\be}(Y_n\cdot e_1)}{n}$ and $v(\be)$ be the speed of $1$-ERW.
\begin{itemize}
\item  The speed $v(\be)$ is infinitely differentiable on $(0,1)$ i.e. $v(\be)\in C^{\infty}(0,1)$ and 
$$\frac{\partial^kv}{\partial\be^k}(\be)=\lim_{n\to\infty}\frac{\partial^kv_n}{\partial\be^k}(\be)\text{ for every }k\in\Nb,\, \be>0.$$ The derivative is expressed in function of renewal time:
\begin{equation}\label{deriv}
\frac{\partial v}{\partial\be}(\be)=\frac{1}{d}.\frac{\hat{\Eb}_{\be}N_{\tau}}{\hat{\Eb}_{\be}\tau}+\frac{\be}{d}.
\frac{\hat{\Eb}_{\be}(N_{\tau}V_{\tau})\hat{\Eb}_{\be}\tau-\hat{\Eb}_{\be}N_{\tau}\hat{\Eb}_{\be}({\tau}V_{\tau})}{(\hat{\Eb}_{\be}\tau)^2}\text{ for }\be>0,
\end{equation}
where $$\E_i=X_{i+1}-X_i,\,V_n=\sum_{i=0}^{n-1}\frac{\E_i1_{Y_i\notin}}{1+\be\E_i}.$$ 
\item For $d=2$ or $d\geq 4$, the speed is differentiable at $\be=0$, the derivative at $0$ is positive, and such that
$$\lim_{\be\to 0}\frac{v(\be)}{\be}=\lim_{n\to\infty}\frac{\partial v_n}{\partial \be}(0),$$
with $$\lim_{n\to\infty}\frac{\partial v_n}{\partial\be}(0)=\frac{1}{d}\lim_{n\to\infty}\frac{R_n(0)}{n}=\frac{1}{d}R(0)>0\text{ for }d\geq 4,\text{ and equals }0\text{ for }d=2.$$
\item For $d=3$ then $$\limsup_{\be\to0}\frac{v(\be)}{\be}\leq\lim_{n\to\infty}\frac{\partial v_n}{\partial\be}(0)=\frac{1}{d}R(0).$$
\end{itemize}
\end{thm}
R. van der Hofstad and M. Holmes proved in  \cite{vdHH10} that the speed of $1$-ERW $v$ is strictly increasing in $\be$ for $d\ge 9$ and is increasing in a neighborhood of $0$ for $d=8$ relying on the lace expansion technique. This technique is also used to prove the monotonicity for random walk in partially random environment ($m=+\infty$), see \cite{HSu12}. Using the
same expansion technique, it is shown in \cite{Hol12}, Th. $2.3$ that the speed
is in an appropriate sense continuous in the drift parameter $\be$ if
$d\geq 6$ and even differentiable if $d\geq 8.$ In our paper, we prove that the speed is infinitely differentiable on $(0,1)$ for all $d\geq 2$ using renewal times and Girsanov's transform. We are also interested in the derivative at the critical point $\be=0$. When the derivative at $0$ is positive and is continuous at $0$ then the velocity is monotonic in a neighborhood of $0$. The existence of the derivative at $0$ of the speed of a random process and the relation of that derivative with the diffusion constant of the
equilibrium state play an important role in mathematical physics. This problem is known as ``Einstein relation for random process", (see for instance the work of N. Gantert, P. Mathieu and A. Piatnitski \cite{GMP12}, see also  \cite{BHOZ13}, \cite{KOa05}, \cite{KOb05}, \cite{LR94}). In \cite{GMP12}, the authors used renewal times by $\tau_n\sim \frac{n}{\lam^2}$ and they used the Markov property of a random process to prove the existence of the derivative of the speed at $0$. However, we have not yet known how to use this technique for $1$-ERW when the Markov property disappears and the increments $\{Y_{[\tau_n,\tau_{n+1})}\}_{n\geq 0}$ are not independent anymore where $\{Y_{[n,n+p)}\}=\{Y_{n+i}-Y_n, 1\leq i< p\}$. In our present work, we use a special method to prove the existence of the derivative at $0$ for any dimension $2\leq d \ne 3.$ The the case $d=3$ is still open. In addition, the question on the existence of the second derivative of the speed at $0$ is also unknown, even the continuity of the derivative of the speed at $0.$

In Theorem \ref{smallD}, we have the formula of the derivative of the speed that is expressed in function of renewal time:
$$
\frac{\partial v}{\partial\be}(\be)=\frac{1}{d}.\frac{\hat{\Eb}_{\be}N_{\tau}}{\hat{\Eb}_{\be}\tau}+\frac{\be}{d}.
\frac{\hat{\Eb}_{\be}(N_{\tau}V_{\tau})\hat{\Eb}_{\be}\tau-\hat{\Eb}_{\be}N_{\tau}\hat{\Eb}_{\be}({\tau}V_{\tau})}{(\hat{\Eb}_{\be}\tau)^2}\text{ for }\be>0.
 $$
We have $$\frac{\partial v}{\partial\be}(0)=\lim_{\be\to 0}\frac{1}{d}.\frac{\hat{\Eb}_{\be}N_{\tau}}{\hat{\Eb}_{\be}\tau}.$$
Then, if we want to prove the continuity of the derivative at $0$, it is equivalent to proving that:
$$\lim_{\be\to 0}\left|\frac{\partial v}{\partial\be}(\be)-\frac{\partial v}{\partial\be}(0)\right
|=\lim_{\be\to 0}\frac{\be}{d}.
\left|\frac{\hat{\Eb}_{\be}(N_{\tau}V_{\tau})\hat{\Eb}_{\be}\tau-\hat{\Eb}_{\be}N_{\tau}\hat{\Eb}_{\be}({\tau}V_{\tau})}{(\hat{\Eb}_{\be}\tau)^2}\right|=0.$$
For this reason, we want to estimate the renewal time $\tau$ as a function of the bias $\be.$
On the other hand, the speed at time $n$ is that: $v_n(\be)=\frac{\be}{d}\Eb_{\be}\left(\frac{N_n}{n}\right).$ Using the Girsanov's transform in Lemma \ref{Girf}, and the notations in \eqref{NoGirf} we obtain 
$$
v_n(\be)=\frac{\be}{d}\Eb_{0}\left(\frac{N_n M_n}{n}\right)
$$
Take the derivative in $\be$, with remark that $N_n=\sum_{i=0}^{n-1}1_{Y_i\notin}$ does not depend on $\be$ we get:
\begin{align*}
\frac{\partial v_n}{\partial\be}(\be)=
\frac{1}{d}\Eb_{0}\left(\frac{N_n M_n}{n}\right)+\frac{\be}{d}\Eb_{0}\left(\frac{N_n\frac{\partial M_n}{\partial\be}}{n}\right).
\end{align*}
Moreover,
$$
\frac{\partial M_n}{\partial\be}=\frac{\partial}{\partial\be}\left[\prod_{i=0}^{n-1}(1+\be\E_i1_{Y_i\notin})\right]=\left(\sum_{i=0}^{n-1}\frac{\E_i1_{Y_i\notin}}{1+\be\E_i}\right)\left[\prod_{i=0}^{n-1}(1+\be\E_i1_{Y_i\notin})\right]=V_nM_n.
$$ 
Therefore, combining with Girsanov's transform we get
$$
\frac{\partial v_n}{\partial\be}(\be)=\frac{1}{d}\Eb_{0}\left(\frac{N_n M_n}{n}\right)+\frac{\be}{d}\Eb_{0}\left(\frac{N_nV_nM_n}{n}\right)
=\frac{1}{d}\Eb_{\be}\left(\frac{N_n}{n}\right)+\frac{\be}{d}\Eb_{\be}\left(\frac{N_nV_n}{n}\right).
$$ 
Then, combining the second point of Theorem \ref{smallD}, for $2\leq d\ne 3$ we obtain
\begin{align}\label{conj}
\limsup_{\be\to 0}\left|\frac{\partial v}{\partial\be}(\be)-\frac{\partial v}{\partial\be}(0)\right|=\limsup_{\be\to 0}\lim_{n\to\infty}\left|\frac{\partial v_n}{\partial\be}(\be)-\frac{\partial v_n}{\partial\be}(0)\right|=\limsup_{\be\to 0}\lim_{n\to\infty}\left|\frac{\be}{d}\Eb_{\be}\left(\frac{N_nV_n}{n}\right)\right|.
\end{align}
Since \eqref{defE},
\begin{align}\label{ec0}
\Eb_{\be}\left(\frac{\E_i1_{Y_i\notin}}{1+\be\E_i}|\F_i\right)&=\frac{\Pb_{\be}(Y_i\notin)}{1+\be}\Pb_{\be}(\E_i=1|\F_i,Y_i\notin)-\frac{\Pb_{\be}(Y_i\notin)}{1-\be}\Pb_{\be}(\E_i=-1|\F_i,Y_i\notin)\\
&=\frac{\Pb_{\be}(Y_i\notin)}{1+\be}\frac{1+\be}{2d}-\frac{\Pb_{\be}(Y_i\notin)}{1-\be}\frac{1-\be}{2d}=0.
\end{align}
This implies that
\begin{align}\label{e0}
\Eb_{\be}\left(\frac{\E_i1_{Y_i\notin}}{1+\be\E_i}\right)=\Eb_{\be}\left[\Eb_{\be}\left(\frac{\E_i1_{Y_i\notin}}{1+\be\E_i}|\F_i\right)\right]=0,\text{ so } \Eb_{\be}(V_n)=0.
\end{align}

Moreover, for all $i>j$ then
\begin{align}\label{em0}
\Eb_{\be}\left(\frac{\E_i1_{Y_i\notin}}{1+\be\E_i}\frac{\E_j1_{Y_j\notin}}{1+\be\E_j}\right)=\Eb_{\be}\left[\frac{\E_j1_{Y_j\notin}}{1+\be\E_j}\Eb_{\be}\left(\frac{\E_i1_{Y_i\notin}}{1+\be\E_i}|\F_i\right)\right]=0.
\end{align}
Combining \eqref{conj} and \eqref{e0} we have
\begin{align*}
\limsup_{\be\to 0}\left|\frac{\partial v}{\partial\be}(\be)-\frac{\partial v}{\partial\be}(0)\right|&=\limsup_{\be\to 0}\lim_{n\to\infty}\frac{\be}{d}\left|\frac{Cov_{\be}(N_n,V_n)}{n}\right|\\
&\leq \limsup_{\be\to 0}\limsup_{n\to\infty}\frac{\be}{d}\sqrt{\frac{Var_{\be}N_n}{n}}\sqrt{\frac{Var_{\be}V_n}{n}}.
\end{align*}

\begin{align*}
\frac{Var_{\be}V_n}{n}=\frac{\Eb_{\be}\left[\left(\sum_{i=0}^{n-1}\frac{\E_i1_{Y_i\notin}}{1+\be\E_i}\right)^2\right]}{n}=\frac{\sum_{i=0}^{n-1}\Eb_{\be}\left[\left(\frac{\E_i1_{Y_i\notin}}{1+\be\E_i}\right)^2\right]}{n}\leq \frac{1}{(1-\be)^2}.
\end{align*}
Therefore, if $$\lim_{\be\to 0}\limsup_{n\to\infty}\left(\be^2\frac{Var_{\be}N_n}{n}\right)=0$$
then the derivative $\frac{\partial v}{\partial\be}(\be)$ is continuous at $0.$

For $m-$ERW, we have some results about the monotonicity as follows:
\begin{thm}\label{monotonesmall} Consider $m$-ERW, $Y=(Y_n)_{n\geq 0}$, let $v(m,\be)$ be the speed of $m$-ERW with bias $\be.$ Then,
\begin{itemize}
\item The range $R(\be)$ of $\infty$-ERW is monotonic in $\be\in[0,1]$.
\item For $d\geq 4$, for every $0<\be_0<\be_1<1$, there exists an integer $m_0=m(\be_0,\be_1)$ large enough such that $v(m,\be)$ is monotonic on $[\be_0,\be_1]$ for every $m\geq m_0.$
\end{itemize}
\end{thm}
In the second part of this theorem, when $d\geq 8$ we recover Theorem $1.2$ of \cite{Pha15}. This result for $d\geq 8$ can also be obtained by minor modification of the proof of Theorem $2.3$ of \cite{Hol12}.
Let us explain the organization of this paper.

In Section $2$ we present the renewal structure for random walks, we also recall some important properties of renewal times.

In Section $3$ we give the proof of Theorem \ref{smallD}. We use Girsanov's transform with renewal time to prove that the speed of $1$-ERW is infinitely differentiable for bias $\be$ positive. For the existence of the derivative at critical point $0$, we use a special method for $1$-ERW.

In Section $4$ we prove Theorem \ref{monotonesmall}. Firstly, we introduce the coupling of $m-$ERW with stationary random walk with bias $\be$. For the monotonicity of the range of simple random walk ($\infty$-ERW) we see that a similar proof can apply for the case random walk in random environment on integer (i.e. in dimension $d=1$). For the dimension $d\geq 2$ we have known yet how to prove this problem. Actually a lot is known about monotonicity of 1-dimensional multi-excited random walks (including RWRE). This includes monotonicity of running maxima, local times, velocity...(see \cite{Zer05},\cite{HSa12}, \cite{Pet13}, \cite{Hol15}), the continuity of the speed was considered in \cite{Zer05},\cite{BS08}. For monotonicity of the speed of $m$-ERW when $m$ large enough, we use the coupling method and a key Lemma presented in this section.
 
\section{The renewal structure}\label{rene}
We define the renewal times for a random walk.
Let $\{Y_n\}_{n\geq 0}$ be a random walk on $\Zb^d$.
\begin{defn}\label{tau}
We present the definition based on the definition given in \cite{BR07} and \cite{MPRV12}. With the convention that $\inf\{\emptyset\}=\infty$ and all random times in the Definition take the value on $[0,+\infty]$. For every $u>0$ let:
$$T_u=\min\{k\geq 1: Y_k\cdot e_1\geq u\}.$$
Define 
$$\bar{D}=\inf\{m\geq 0: Y_m\cdot e_1<Y_0\cdot e_1\}.$$
Furthermore, define two sequences of $\F^Y_n-$stopping times $\{S_n: n\geq 0\}$ and $\{D_n: n\geq 0\}$ as follows: We let $S_0=0, R_0=Y_0\cdot e_1$ and $D_0=0.$ Next, define by induction on $k\geq 0$
\begin{align*}
&S_{k+1}=T_{R_k+1}\\
&D_{k+1}=\bar{D}\circ\th_{S_{k+1}}+S_{k+1}\\
&R_{k+1}=\sup\{Y_{i}\cdot e_1: 0\leq i\leq D_{k+1}\},
\end{align*}
where $\th$ is the canonical shift on the space of trajectories. Let
$$\ka=\inf\{n\geq 0: S_n<\infty, D_n=\infty\}.$$ We define the first renewal time as follows: $$\tau_1=S_{\ka}.$$
We then define by induction on $n\geq 1$, the sequence of renewal times $\tau_1, \tau_2,...$ as follows:
$$\tau_{n+1}=\tau_n+\tau_1(Y_{\tau_n+\cdot}).$$
Next, we define $D_i^{(0)}=D_i$ and $S_i^{(0)}=S_i$ and for every $k\geq 1$ two sequences $D_i^{(k)}$ and $S_i^{(k)}$ w.r.t. the trajectory $(Y_{\tau_{k}+.})$, of the same way that the sequences $D_i$ and $S_i$ are defined w.r.t. $(Y.)$
For example, $S_0^{(1)}, R_0^{(1)}=Y_{\tau_1}\cdot e_1, D_0^{(1)}=0$ and we define by induction on $i\geq 0,$
\begin{align*}
&S_{i+1}^{(1)}=T_{R_i^{(1)}+1}\\
&D_{i+1}^{(1)}=\bar{D}\circ\th_{S_{i+1}^{(1)}}+S_{i+1}^{(1)}\\
&R_{i+1}^{(1)}=\sup\{Y_i\cdot e_1: 0\leq i\leq D_{i+1}^{(1)}\}.
\end{align*}
For every $k\geq 1$ and $j\geq 0$ such that $S_j^{(k)}<\infty,$ we need to introduce the $\si$-algebra $\G_j^{(k)}$ of the events up to $S_j^{(k)}$ as the smallest $\si-$algebra containing all of the sets of the form $\{\tau_1\leq n_1\}\cap\{\tau_2\leq n_2\}\cap...\{\tau_k\leq n_k\}\cap A$, where $n_1<n_2<...<n_k$ are integers and $A\in\F_{n_k+S_j^{(0)}\circ \th_{n_k}}.$ Let $\tau_0=0$ and  $G_0^{(0)}$ be trivial.
\end{defn}

The signification of renewal times is given by the following lemma:
\begin{lem}\label{esstau}
The first renewal time $\tau_1$ is the first time when the random walk attains the hyperplane $\{y\cdot e_1=Y_{\tau_1}\cdot e_1\}$, and after that it does not come back anymore behind this hyperplane, i.e.
$$\tau_1=\inf\{n\geq 0: \sup_{0\leq i<n}Y_i\cdot e_1<Y_n\cdot e_1\leq \inf_{n\leq i } Y_i\cdot e_1\}.$$
\end{lem}
\begin{proof} By the definition of two sequences $S_i$ and $D_i$ we have: $S_0=D_0=0<S_1<D_1<S_2<D_2<....$
Because $D_i,S_i$ are integers, then $\lim_{i\to\infty}S_i=\lim_{i\to\infty}D_i=+\infty.$ Let $$\tau'_1=\inf\{n\geq 0: \sup_{0\leq i<n}Y_i\cdot e_1<Y_n\cdot e_1\leq \inf_{n\leq i } Y_i\cdot e_1\},$$
we prove that $\tau_1=\tau'_1.$ Firstly, it is clear that $\tau_1\in\{n\geq 0: \sup_{0\leq i<n}Y_i\cdot e_1<Y_n\cdot e_1\leq \inf_{n\leq i } Y_i\cdot e_1\}$ so $\tau'_1\leq \tau_1.$ On the other hand, there exists an integer $i_0$ such that $S_{i_0}\leq \tau'_1<S_{i_0+1}.$ We prove that $\tau'_1=S_{i_0}.$ In fact, if $S_{i_0}< \tau'_1\leq D_{i_0}$ then by the definition of $\tau'_1,$ we have $Y_{S_{i_0}}\cdot e_1<Y_{D_{i_0}}\cdot e_1$, this is contrary to the fact that $Y_{S_{i_0}}\cdot e_1>Y_{D_{i_0}}\cdot e_1.$ If $D_{i_0}< \tau'_1< S_{i_0+1}$, then $R_{i_0}=\sup\{Y_{i}\cdot e_1: 0\leq i\leq D_{i_0}\}<Y_{\tau'_1}\cdot e_1$ and $R_{i_0}+1\leq Y_{\tau'_1}\cdot e_1.$ Hence $\tau'_1\in\{k\geq 1: Y_k\cdot e_1\geq R_{i_0}+1\}$ and by the definition of the sequence $\{S_i\}$, it implies that $S_{i_0+1}\leq \tau'_1$. This contradicts the supposition that $D_{i_0}< \tau'_1< S_{i_0+1}.$ So, it remains only $\tau'_1=S_{i_0.}$ By the definition of $\tau'_1$ we obtain $\bar{D}\circ\th_{S_{i_0}}=\infty$, then $D_{i_0}=\infty.$ This implies that $i_0\geq\ka=\inf\{n\geq 0: S_n<\infty, D_n=\infty\}$, and $\tau_1=S_{\ka}\leq S_{i_0}=\tau'_1.$
\end{proof}
On the existence of renewal times and the existence of the moments of all orders for $1$-ERW, we have the following key lemma proved in \cite{BR07}, \cite{MPRV12}:  
\begin{lem}\label{esttau}
Consider a $1$-ERW with bias $\be$, let $(\tau_k,k\geq 1)$ be the associated renewal times. Then, there exists $C,\al>0$ depending on $\be$ and such that for every $n\geq 1,$
$$\sup_{k\geq 0}\Pb_{\be}[\tau_{k+1}-\tau_k>n|\G_0^{(k)}]\leq Ce^{-n^{\al}}\; a.s.$$
In particular, for every $k\geq 0$ and $p\geq 1$, then $\tau_k<\infty\,, a.s.$ and $\Eb_{\be}[(\tau_{k+1}-\tau_k)^p]<\infty.$
\end{lem}
The lemma above gives an estimation of renewal times for every parameter $\be$ fixed. We know that, when $\be=0$, there does not exist the renewal times. We would like to estimate the renewal times when $\be$ converges to $0$. But this is an interesting and difficult question. The method using renewal times is used in many models to prove the law of large numbers and to prove the Einstein's relation. This problem in mathematical physic that is studied the first time by the greatest physician Albert Einstein, see \cite{Ein1905}. Recently this problem appears in the works of mathematicians, for example in \cite{BHOZ11}, \cite{GMP12},... Einstein's relation means to study the relation between the diffusion constant at the equilibrium state and the derivative of the speed of stochastic process at the critical point w.r.t. the balanced state. 

A property very important of renewal times is that, they cut a trajectory of the random walk into the independent increments as the following lemma (see \cite{BR07} and \cite{MPRV12}):
{\allowdisplaybreaks
\begin{lem}\label{indtau}
Under the probability $\Pb_{\be}$, the random variables $(X_{\tau_{k+1}}-X_{\tau_k},\tau_{k+1}-\tau_k)_{k\geq 1}$ and $(X_{\tau_1},\tau_1)$ are independent and $(X_{\tau_{k+1}}-X_{\tau_k},\tau_{k+1}-\tau_k)_{k\geq 1}$ have the same law of $(X_{\tau_1},\tau_1)$ under the probability $\Pb_{\be}$ conditionally on $\bar{D}=\infty$, write $\hat{\Pb}_{\be}(\cdot)=\Pb_{\be}(\cdot|\bar{D}=\infty).$
\end{lem}
From Lemmas \ref{esttau} et \ref{indtau} and with the notation $\tau=\tau_1$, we have $\Pb_{\be}[\tau\geq n]<Ce^{-n^{\al}}.$
} Note that Lemma \ref{indtau} is not anymore true for the model of generalized excited random walk (see \cite{MPRV12}), and also for the case the definition of renewal times is modified as in \cite{She03}, \cite{GMP12}. We want to estimate the moments of $\tau$ as a function of $\be$ by the question that there exists an integer $k$ such that $\sup_{\be\in(0,1]}\be^k\hat{\Eb}_{\be}\tau<\infty$, or a better estimation $\sup_{\be\in(0,1]}\be^{2k}\hat{\Eb}_{\be}\tau^2<\infty$? We are interested in the case $k=2$, and we would like to find a definition of renewal times to obtain the estimation of $k=2.$ With the definition \ref{tau}, it is difficult to estimate $\tau$. It is useful to change a little the definition of $\tau$. For example, in \cite{GMP12} the random walk is allowed to come back behind the hyperplane (in Lamma \ref{esstau}) of a distant $\lam=\frac{\ep}{\be}$, this means that we redefine $$\bar{D}=\inf\{m\geq 0: Y_m\cdot e_1<Y_0\cdot e_1-\lam\}.$$ With this change, for Markov process, Lemma \ref{indtau} is still true, but for $1$-ERW and non-Markovian process, it is not anymore true. This is a difficulty when we want to study non-Markovian process by using renewal times.

\section{Proof of Theorem \ref{smallD}}
We repeat some necessary notations .
\begin{itemize}
\item $(Y_n)_{n\in\Zb}$ are the cordinate maps on $\Zb^d$ and $\Pb_{\be}$ is the law of $1$-ERW. The speed is $v=v(\be),$ 
\item $\{\tau_n\}$ is the sequence of renewal times,
\item $X_n=Y_n\cdot e_1,\ Z_n=(Y_n\cdot e_2,Y_n\cdot e_3,...,Y_n\cdot e_d),\ \E_n=X_{n+1}-X_n$,
\item $\bar{\E}_n=\E_n-\Eb_{\be} \E_n,\ \E'_n=\E_n-v,\ V_n=\sum_{j=0}^{n-1}\frac{\E_j1_{Y_j\notin}}{1+\be\E_j},$
\item The speed at the time $n$, the speed of $1$-ERW and the derivative of the speed at the time $n$ respectivly are $$v_n(\be)=\Eb_{\be}\left(\frac{X_n}{n}\right),\ v(\be)=\frac{\hat{\Eb}_{\be} X_{\tau}}{\hat{\Eb}_{\be}{\tau}},\text{ and } \frac{\partial v_n}{\partial \be}=\frac{\Eb_{\be} (X_nV_n)}{n},$$
\item $\bar{X}_n=\sum_{j=0}^{n-1}\bar{\E}_j,\ X'_n=\sum_{j=0}^{n-1}\E'_j,\ a=\hat{\Eb}_{\be}\tau.$
\end{itemize}
\subsection{The existence of the derivative of the speed for $\be>0$}
\begin{rem}
Using renewal times, we have that
$$\lim_{n\to\infty}v_n(\be)=v(\be),\ 0=\lim_{n\to\infty}\Eb_{\be}\left(\frac{\bar{X}_n}{n}\right)=\frac{\hat{\Eb}_{\be}\bar{X}_{\tau}}{\hat{\Eb}_{\be}\tau}=\frac{\hat{\Eb}_{\be}(X_{\tau}-\sum_{j=0}^{\tau-1}\Eb_{\be}\E_j)}{\hat{\Eb}_{\be}\tau},$$
$$\Pb_{\be}-a.s.\text{ then }\,\lim_{n\to\infty}\frac{X'_n}{n}=\frac{\hat{\Eb}_{\be}X'_{\tau}}{\hat{\Eb}_{\be}\tau}=\frac{\hat{\Eb}_{\be}(X_{\tau}-v\tau)}{\hat{\Eb}_{\be}\tau}=0.$$
\end{rem}
We deduce from these equalities that $\hat{\Eb}_{\be}\bar{X}_{\tau}=0$ and $\hat{\Eb}_{\be}{X}_{\tau}=\hat{\Eb}_{\be}[\sum_{j=0}^{\tau-1}(\Eb_{\be}\E_j)].$
\subsubsection{The existence of the limits of the derivatives at finite times}
To prove the point $1$ of Theorem \ref{smallD}, we need the following lemmas:
\begin{lem}\label{Lem31}
\begin{align}
&\sup_{n\geq 1}\frac{\Eb_{\be}\left(\max_{0\leq i\leq n}{|X'_i|^2}\right)}{n}:=C_1(\be)<+\infty\\
&\sup_{n\geq 1}\frac{\Eb_{\be}\left(\max_{0\leq i\leq n}|\bar{X}_i|^2\right)}{n}:=C_2(\be)<+\infty
\end{align}
\end{lem}
\begin{proof}
Firstly, we prove that 
$$\sup_{n\geq 1}\frac{\Eb_{\be}\left(\max_{0\leq i\leq [na]}{|X'_i|^2}\right)}{n}:=C'_1(\be)<+\infty\text{ where }a=\hat{\Eb}_{\be}\tau.$$
Let $S'_i=X'_{\tau_i}$, then
$$\max_{0\leq i\leq [na]}{{X'_i}^2}\leq\max_{0\leq i\leq [na]}{{S'_i}^2}+(\tau_n-[na])^2+\sum_{j=0}^{n-1}(\tau_{j+1}-\tau_j)^2.$$
Because that $\max_{0\leq i\leq [na]}{{X'_i}^2}$ attains max at $i_0$ then either $i_0\in[\tau_n,[na]]$ or there exists $j_0$  such that $i_0\in[\tau_{j_0},\tau_{j_0+1}).$
Since $$(\tau_n-[na])^2=[(\tau_n-na+na-[na])]^2\leq 2[(\tau_n-na)^2+1],$$ we get
$$\Eb_{\be}\left(\max_{0\leq i\leq [na]}{|X'_i|^2}\right)\leq\max_{0\leq i\leq n}{S'_i}^2+2\Eb_{\be}(\tau^2)+2(n-1)\hat{\Eb}_{\be}(\tau-a)^2+(n-1)\hat{\Eb}_{\be}(\tau^2)+2.$$
Note that $\{S'_i\}$ is the martingale then 
$$\Eb_{\be}(\max_{0\leq i\leq n}{{S'_i}^2})\leq 4\Eb_{\be}\left({S'_n}^2\right)=4\Eb_{\be}({X'_{\tau}}^2)+4(n-1)\hat{\Eb}_{\be}({X'_\tau}^2)\leq 4\Eb_{\be}({{\tau}}^2)+4(n-1)\hat{\Eb}_{\be}({\tau}^2) .$$
Therefore,
$$\sup_{n\geq 1}\frac{\Eb_{\be}\left(\max_{0\leq i\leq [na]}{|X'_i|^2}\right)}{n}\leq\sup_{n\geq 1}\left( \frac{4n+2}{n}\hat{\Eb}_{\be}({\tau}^2)+ \frac{2n-2}{n}\hat{\Eb}_{\be}(\tau-a)^2+\frac{2}{n}\right)<+\infty.$$
We now consider the sequence of integers $\{p_n\}$ such that 
$[p_na]\leq n<[(p_n+1)a]$ then $n/p_n\to a$.
We deduce that
$$\sup_{n\geq 1}\frac{\Eb_{\be}\left(\max_{0\leq i\leq n}{|X'_i|^2}\right)}{n}\leq \sup_{n\geq 1}\frac{\Eb_{\be}\left(\max_{0\leq i\leq [(p_n+1)a]}{|X'_i|^2}\right)}{(p_n+1)}\times \frac{(p_n+1)}{n}\leq \infty.$$
It is similar to prove that
$$\sup_{n\geq 1}\frac{\Eb_{\be}\left(\max_{0\leq i\leq n}{\bar{X}_i^2}\right)}{n}=C_2(\be)<+\infty;$$
$$\sup_{n\geq 1}\frac{\Eb_{\be}\left(\max_{0\leq i\leq n}{V_i^2}\right)}{n}=C_3(\be)<+\infty.$$
\end{proof}
\begin{lem}
\begin{align}
&\sup_{n,p\geq 1}\frac{\Eb_{\be}\left(\max_{0\leq i\leq p}{\left({X'_{\tau_n+i}-X'_{\tau_n}}\right)^2}\right)}{p}=C_4(\be)<+\infty\\
&\sup_{n\geq 1}\sup_{0<p\leq [n/a]}\frac{\Eb_{\be}\left(\max_{0\leq i\leq p}{\left({X'_{\tau_n}-X'_{\tau_n-i}}\right)^2}\right)}{p}=C_5(\be)<+\infty\label{Lem32}.
\end{align}
We have similarly the result for the sequences $\{\bar{X}_n\}$ and $\{V_n\}$.
\end{lem}
\begin{proof}
From Lemma \ref{Lem31} we get 
\begin{align*}
&\sup_{n\geq 1}\frac{\hat{\Eb}_{\be}\left(\max_{0\leq i\leq n}{{X'_i}^2}\right)}{n}\leq\sup_{n\geq 1}\frac{\Eb_{\be}\left(\max_{0\leq i\leq n}{{X'_i}^21_{D=0}}\right)}{n\Pb(D=0)}\\
&\leq \frac{1}{\Pb(D=0)}\sup_{n\geq 1}\frac{\Eb_{\be}\left(\max_{0\leq i\leq n}{{X'_i}^2}\right)}{p}<+\infty.
\end{align*}
Therefore
\begin{align*}
\sup_{n,p\geq 1}\frac{\Eb_{\be}\left(\max_{0\leq i\leq p}{\left({X'_{\tau_n+i}-X'_{\tau_n}}\right)^2}\right)}{n}=\sup_{n\geq 1}\sup_{p\geq 1}\frac{\hat{\Eb}_{\be}\left(\max_{0\leq i\leq p}{\left({X'_{i}}\right)^2}\right)}{n}
\end{align*}
To prove (\ref{Lem32}), we consider
$$\sup_{0<p\leq [n/a]}\frac{\Eb_{\be}\left(\max_{0\leq i\leq [pa]}{\left({X'_{\tau_n}-X'_{\tau_n-i}}\right)^2}\right)}{p}$$
For $0<p\leq [n/a]$ then $[pa]\leq pa\leq (n/a).a=n.$ Set $S'_i=X'_{\tau_n}-X'_{\tau_{n-i}}$ so that
$$\max_{0\leq i\leq [pa]}{\left({X'_{\tau_n}-X'_{\tau_n-i}}\right)^2}\leq \max_{0\leq i\leq p}{S'_i}^2+\sum_{j=0}^{p-1}(\tau_{n-j}-\tau_{n-j-1})^2+(\tau_n-\tau_{n-p}-[pa])^2.$$
We deduce from the inequality above that
{\allowdisplaybreaks
\begin{align*}
&\Eb_{\be}\left[\max_{0\leq i\leq [pa]}{\left({X'_{\tau_n}-X'_{\tau_n-i}}\right)^2}\right]\\
&\leq \Eb_{\be}\left(\max_{0\leq i\leq p}{S'_i}^2\right)+\sum_{j=0}^{p-1}\Eb_{\be}(\tau_{n-j}-\tau_{n-j-1})^2+\Eb_{\be}(\tau_n-\tau_{n-p}-[pa])^2\\
&\leq 4\Eb_{\be}({S'_p}^2)+p\hat{\Eb}_{\be}(\tau^2)+2p\Eb_{\be}[(\tau-a)^2]+2\\
&\leq 4p\hat{\Eb}_{\be}({X'_{\tau}}^2)+p\hat{\Eb}_{\be}(\tau^2)+2p\Eb_{\be}[(\tau-a)^2]+2.
\end{align*}}
Therefore,
$$\sup_{p\geq 1}\frac{\Eb_{\be}\left(\max_{0\leq i\leq [pa]}{\left({X'_{\tau_n}-X'_{\tau_n-i}}\right)^2}\right)}{p}:=C(\be)<\infty.$$
Let $p\leq [n/a]$ and $0\leq i\leq p$, there exists a sequence $\{p_n\}$ such that $[p_na]< p\leq[(p_n+1)a]$. Because that $a\geq 1$ and $[p\leq [\frac{n}{a}]\leq[[\frac{n}{a}]a]\leq \frac{n}{a}.a=n$ then $[(p_n+1)a]\leq[[\frac{n}{a}]a]\leq n.$ 
So, we have
\begin{align*}
&\sup_{p\geq 1}\frac{\Eb_{\be}\left(\max_{0\leq i\leq [(p_n+1)a]}{\left({X'_{\tau_n}-X'_{\tau_n-i}}\right)^2}\right)}{p}\\
& \leq\sup_{p\geq 1}\frac{\Eb_{\be}\left(\max_{0\leq i\leq [(p_n+1)a]}{\left({X'_{\tau_n}-X'_{\tau_n-i}}\right)^2}\right)}{p_n+1}.\frac{p_n+1}{p}\\
&\leq C(\be).\frac{p+a}{ap}\leq C'(\be)<+\infty.
\end{align*}\end{proof}
\begin{lem}
$\lim_{n\to+\infty}\left|\frac{1}{n}\Eb_{\be}(X'_{\tau_n}V_{\tau_n})-\frac{1}{n}\Eb_{\be}(X'_{[na]}V_{[na]})\right|=0.$
\end{lem}
\begin{proof}
Using the inequality $|(a+\de)(b+\de)-ab|\leq |a\de|+|b\de|+|\de^2|$, we have that
\allowdisplaybreaks
{
\begin{align*}
&\left|\frac{1}{n}\Eb_{\be}(X'_{\tau_n}V_{\tau_n})-\frac{1}{n}\Eb_{\be}(X'_{[na]}V_{[na]})\right|
\leq\frac{1}{n}\Eb_{\be}\left|\sum_{j=[na]}^{\tau_n-1}\E'_j\right|.|V_{\tau_n}|\\&+\frac{1}{n}\Eb_{\be}\left(|X'_{\tau_n}|.\left|\sum_{j=[na]}^{\tau_n-1}\frac{\E_j1_{Y_j\notin}}{1+\be\E_j}\right|\right)+\frac{1}{n}\Eb_{\be}\left|\left(\sum_{j=[na]}^{\tau_n-1}\E'_j\right).\left(\sum_{k=[na]}^{\tau_n-1}\frac{\E_j1_{Y_j\notin}}{1+\be\E_j}\right)\right|\\
&\leq \sqrt{\frac{1}{n}\Eb_{\be}\left[\left(\sum_{j=[na]}^{\tau_n-1}\E'_j\right)^2\right]}.\sqrt{\frac{1}{n}\Eb_{\be}[(V_{\tau_n})^2}+\sqrt{\frac{1}{n}\Eb_{\be}\left[\left(\sum_{j=[na]}^{\tau_n-1}\frac{\E_j1_{Y_j\notin}}{1+\be\E_j}\right)^2 \right]}.\sqrt{\frac{1}{n}\Eb_{\be}({X'_{\tau_n}}^2)}\\
&+\sqrt{\frac{1}{n}\Eb_{\be}\left[\left(\sum_{j=[na]}^{\tau_n-1}\E'_j\right)^2\right]}.\sqrt{\frac{1}{n}\Eb_{\be}\left[\left(\sum_{j=[na]}^{\tau_n-1}\frac{\E_j1_{Y_j\notin}}{1+\be\E_j}\right)^2 \right]}
\end{align*}
}
There exist two finite constants $C(\be),C'(\be)$ depending only on $\be$ such that
\begin{itemize}

\item$\text{For all } n\geq 1 \text{ then } \frac{1}{n}\Eb_{\be}(V_{\tau_n}^2)=\frac{1}{n}\Eb_{\be}(V_{\tau}^2)+\frac{n-1}{n}\hat{\Eb}_{\be}(V_{\tau}^2)\leq C(\be);$
\item$\text{For all }n\geq 1\text{ then } \frac{1}{n}\Eb({X'_{\tau_n}}^2)=\frac{1}{n}\Eb_{\be}({X'_{\tau}}^2)+\frac{n-1}{n}\hat{\Eb}_{\be}({X'_{\tau}}^2)\leq C'(\be).$
\end{itemize}
We need prove that $$\lim_{n\to+\infty}\frac{1}{n}\Eb_{\be}\left[\left(\sum_{j=[na]}^{\tau_n-1}\E'_j\right)^2\right]=0.$$
In fact, we have that
\begin{align*}
\frac{1}{n}\Eb_{\be}\left[\left(\sum_{j=[na]}^{\tau_n-1}\E'_j\right)^2\right]
&\leq\frac{1}{n}\Eb_{\be}\left[(\tau_n-[na])^21_{|\tau_n-[na]|\geq \ep n}\right]+\frac{1}{n}\Eb_{\be}\left[\left(\sum_{j=[na]}^{\tau_n-1}\E'_j\right)^21_{|\tau_n-[na]|< \ep n}\right]\\
&=L_1+L_2.
\end{align*}
Here, $L_1,L_2$ are respectively the first and the second terms of the site on the right hand.
Estimate two terms to get
\begin{align*}
&L_1\leq\frac{2}{n}\Eb_{\be}\left[[(\tau_n-na)^2+1]1_{|\tau_n-[na]|\geq \ep n}\right]\\
&\leq\sqrt{\frac{1}{n^2}\Eb_{\be}[(\tau_n-na)^4].\Pb\left(|\tau_n-[na]|\geq \ep n\right)}+\frac{2}{n}\Pb\left(|\tau_n-[na]|\geq \ep n\right).
\end{align*}
Because $\sup_{n\geq 1}\frac{1}{n^2}\Eb_{\be}[(\tau_n-na)^4]<+\infty$ and $\lim_{n\to+\infty}\Pb\left(|\tau_n-[na]|\geq \ep n\right)=0$ then $\lim_{n\to+\infty}$ $L_1=0.$
On the other hand
\begin{align*}
L_2&\leq\ep.\frac{\Eb_{\be}\left[\max_{0\leq i\leq\ep n}(X'_{\tau_n}-X'_{\tau_n-i})^2+\max_{0\leq i\leq\ep n}(X'_{\tau_n+i}-X'_{\tau_n})^2\right]}{\ep n}\\
&\leq\ep C_4(\be).
\end{align*}
For all $\si>0$ choose $\ep=\frac{\si}{C_4(\be)}$ so that $L_2\leq \si.$ \\
Then we get
$$\limsup_{n\to+\infty}\frac{1}{n}\Eb_{\be}\left[\left(\sum_{j=[na]}^{\tau_n-1}\E'_j\right)^2\right]\leq \si\text{ for all }\si>0.$$
Therefore
$$\lim_{n\to+\infty}\frac{1}{n}\Eb_{\be}\left[\left(\sum_{j=[na]}^{\tau_n-1}\E'_j\right)^2\right]=0.$$
It is similar to prove that
$$\lim_{n\to+\infty}\frac{1}{n}\Eb_{\be}\left[\left(\sum_{j=[na]}^{\tau_n-1}\frac{\E_j1_{Y_j\notin}}{1+\be\E_j}\right)^2 \right]=0.$$
This finishes the proof of Lemma. 
\end{proof}
\begin{cor}
$$\lim_{n\to+\infty}\Eb_{\be}\left[\frac{X'_{[na]}V_{[na]}}{n}\right]=\lim_{n\to+\infty}\Eb_{\be}\left[\frac{X'_{\tau_n}V_{\tau_n}}{n}\right]=\hat{\Eb}_{\be}(X'_{\tau}V_{\tau}).$$
\end{cor}
We now prove the existence of the limit $\frac{\partial v_n}{\partial \be}(\be)$. Let $\{p_n\}$ be the sequence such that $[p_na]\leq n\leq[(p_n+1)a]$ then $\lim_{n\to+\infty}\frac{n}{p_n}=a.$ So, we have 
\begin{align*}
&\left|\Eb_{\be}\left(\frac{X'_nV_n}{n}-\frac{X'_{[p_na]}V_{[p_na]}}{n}\right)\right|\leq \frac{(n-[p_na])^2}{n}+\frac{|n-[p_na]|}{n}.\Eb_{\be}|X'_n|+\frac{|n-[p_na]|}{n}.\Eb_{\be}|V_n|\\
&\leq \frac{a^2}{n}+a.\Eb_{\be}\left|\frac{X'_n}{n}\right|+a.\Eb_{\be}\left|\frac{V_n}{n}\right|.
\end{align*}
When $n$ goes to infinitely then $\frac{X'_n}{n}$ and $\frac{V_n}{n}$ go to $0.$ So that
\begin{align*}
\lim_{n\to+\infty}\Eb_{\be}\left(\frac{X'_nV_n}{n}\right)&=\lim_{n\to+\infty}\Eb_{\be}\left(\frac{X'_{[p_na]}V_{[p_na]}}{n}\right)=\lim_{n\to+\infty}\Eb_{\be}\left(\frac{X'_{[p_na]}V_{[p_na]}}{p_n}\right).\frac{p_n}{n}\\
&=\hat{\Eb_{\be}}(X'_{\tau}V_{\tau}).\frac{1}{a}=\frac{\hat{\Eb}_{\be}(X'_{\tau}V_{\tau})}{\hat{\Eb}_{\be}\tau}=\frac{\hat{\Eb}_{\be}[(X_{\tau}-\tau v)V_{\tau}]}{\hat{\Eb}_{\be}\tau}.
\end{align*}
Therefore,
\begin{equation}\label{limder}
\lim_{n\to+\infty}\frac{\partial v_n}{\partial \be}(\be)=\lim_{n\to+\infty}\Eb_{\be}\left(\frac{X_nV_n}{n}\right)=\lim_{n\to+\infty}\Eb_{\be}\left(\frac{X'_nV_n}{n}\right)=\frac{\hat{\Eb}_{\be}[(X_{\tau}-\tau v)V_{\tau}]}{\hat{\Eb}_{\be}\tau}.
\end{equation}
\subsubsection{Girsanov transform}

In this section we prove the smoothness of the  speed using the Girsanov's transform.
Firstly, we need a lemma as follows:
\begin{lem}\label{estimatau}
For all $c\in(0,1]$ then 
\begin{align*}
&\sup_{t\in[c,1]}\Pb_t(\tau>n)\leq C'e^{n^{-\al}},\\
&\sup_{t\in[c,1]}\hat{\Pb}_t(\bar{D}=\infty)\geq \phi>0,\\
&\sup_{t\in[c,1]}\hat{\Pb}_t(\tau>n)\leq Ce^{n^{-\al}}.
\end{align*}
Where $C',C,\phi,\al$ are positive constants depending only on $c.$
\end{lem} 
\begin{proof}
To prove this lemma, repeating the proof of Proposition 2.1 and Proposition 4.3 of \cite{MPRV12}. Note that for $1$-ERW with the law $\Pb_t, t\geq c$, we can choose the constants $\lam,h,r$ as in \cite{MPRV12} to consider $1$-ERW as a generalized excited random walk such that these parameters depend only on $c$. For more details on the conditions in \cite{MPRV12}, the condition B, we choose $K=1.$ For the condition C$^+$, choose $l=e_1$, 
$$
\Eb_t(Y_{n+1}-Y_n|\F_n)\cdot e_1\geq \frac{1+t}{2d}-\frac{1-t}{2d}=\frac{t}{d}\geq\frac{c}{d}.
$$
Then, choose $\lambda=\frac{c}{d}.$ For the condition E, let $\Sb^{d-1}=\{x\in\Rb^d:||x||=1\}.$ Let $l'\in\Sb^{d-1}$ then $l'=(x_1,x_2,...,x_d)$ with $\sum_{i=1}^d x_i^2=1.$ There exist $i_0\in\{1,2,...,d\}$ such that $|x_{i_0}|=\max_{i}|x_i|$. Hence 
$$
dx_{i_0}^2\geq 1\Leftrightarrow |x_{i_0}|\geq\frac{1}{\sqrt{d}}. 
$$
Consider the function signum:
\begin{equation*}
sgn(x)=
\begin{cases}
+1 &\text{ if }x>0\\
\,\,\,\, 0 &\text{ if }x=0\\
-1 &\text{ if }x<0\\
\end{cases}.
\end{equation*}
Choose $r=\frac{1}{2\sqrt{d}}$ then  $sgn(x_{i_0}) e_{i_0}\cdot l'\geq|x_{i_0}|>r.$ Therefore, on $\{\Eb_{t}(Y_{n+1}-Y_n|\F_n)=0\}$ we get 
$$
\Pb_{t}[(Y_{n+1}-Y_n)\cdot l'>r|\F_n]\geq\Pb_{t}(Y_{n+1}-Y_n=sgn(x_{i_0}) e_{i_0}|\F_n)
\geq \frac{1}{2d}.
$$
Moreover,
$$
\Pb_{t}[(Y_{n+1}-Y_n)\cdot e_1>r|\F_n]\geq 
\Pb_{t}[Y_{n+1}-Y_n= e_1|\F_n]=\frac{1+t}{2d}1_{Y_n\notin}+\frac{1}{2d}1_{Y_n\in}\geq\frac{1}{2d}.
$$
Then choose $h=\frac{1}{2d}$. All of parameters $K,\lambda,r,h$ depend only on $c.$ 
\end{proof}

Let $\be_0,\be\in(0,1]$ we have:
\begin{lem}\label{Girf}
$$
\frac{d\Pb_{\be}}{d\Pb_0}|_{\F_n}=\prod_{i=0}^{n-1}(1+\be\E_i1_{Y_i\notin})
$$
$$
\frac{d\Pb_{\be}}{d\Pb_{\be_0}}|_{\F_n}=\prod_{i=0}^{n-1}\left(\frac{1+\be\E_i1_{Y_i\notin}}{1+\be_0\E_i1_{Y_i\notin}}\right)
$$
\end{lem}
We denote
\begin{equation}\label{NoGirf}
M_n(\be):=\prod_{i=0}^{n-1}(1+\be\E_i1_{Y_i\notin})
\text{ and }
M_n(\be,\be_0):=\prod_{i=0}^{n-1}\left(\frac{1+\be\E_i1_{Y_i\notin}}{1+\be_0\E_i1_{Y_i\notin}}\right).
\end{equation}

To prove the existence of the speed we need the following lemma
\begin{lem}Consider a $\si-$algebra $\F_{\tau}$ that is defined by
$$
\F_{\tau}=\{A\in\F: \forall n,\ \exists B_n\in\F_n \text{ such that }A\cap\{\tau=n\}=\B_n\cap\{\tau=n\}\}.
$$
then $\tau$ is $\F_{\tau}-$measurable, $(\bar{D}=\infty)\in\F_{\tau}$ and
\begin{equation}\label{Girtau}
\frac{d\Pb_{\be}}{d\Pb_{\be_0}}|_{\F_{\tau}}=M_{\tau}(\be,\be_0).\frac{\Pb_{\be}(\bar{D}=\infty)}{\Pb_{0}(\bar{D}=\infty)}
\end{equation}
\end{lem}
\begin{proof}
We see that $(\tau=n)=\Om\cap(\tau=n)$ and $\Om\in\F_n$ for all $n$ then by definition of $\F_{\tau}$ we have $(\tau=n)\in\F_{\tau}$, it means that $\tau$ is $\F_{\tau}-$measurable. It is clear that $(\bar{D}=\infty)=(D\geq \tau)$ so that $(\bar{D}=\infty)\cap(t=n)=(D\geq n)\cap(\tau=n)$. Because that $(D\geq n)\in\F_n$ then we deduce $(\bar{D}=\infty)\in\F_{\tau}.$
Now we prove \ref{Girtau}, for all $A\in\F_{\tau}$ then
\begin{align*}
\Pb_{\be}(A)&=\sum_{n=1}^{\infty}\Pb(A,\tau=n)=\sum_{n=1}^{\infty}\Pb(B_n,\tau=n)=\sum_{n=1}^{\infty}\sum_{\om_n\in B_n}\Pb(\om_n,\tau=n)\\
&=\sum_{n=1}^{+\infty}\sum_{\om_n\in B_n}1_{\om_n,\tau=n}M_n(\be)(\om_n)\Pb_{\be}(\bar{D}=\infty)=\sum_{n=1}^{+\infty}\sum_{\om_n\in B_n}1_{\om_n,\tau=n}M_n(\be)(\om)\Pb_{\be}(\bar{D}=\infty)\\
&=\sum_{n=1}^{+\infty}\sum_{\om_n\in B_n}1_{\om_n,\tau=n}M_n(\be_0)(\om)\Pb_{\be_0}(\bar{D}=\infty).M_{\tau}(\be,\be_0)d\Pb_{\be_0}\frac{\Pb_{\be}(\bar{D}=\infty)}{\Pb_{\be_0}(\bar{D}=\infty)}\\
&=\sum_{n=1}^{+\infty}\int_{B_n,\tau=n}M_{\tau}(\be,\be_0)d\Pb_{\be_0}\frac{\Pb_{\be}(\bar{D}=\infty)}{\Pb_{\be_0}(\bar{D}=\infty)}=\Eb_{\be_0}[1_AM_{\tau}(\be,\be_0)d\Pb_{\be_0}].\frac{\Pb_{\be}(\bar{D}=\infty)}{\Pb_{\be_0}(\bar{D}=\infty)}.
\end{align*}
So, we get
$$
\frac{d\Pb_{\be}}{d\Pb_{\be_0}}|_{\F_{\tau}}=M_{\tau}(\be,\be_0).\frac{\Pb_{\be}(\bar{D}=\infty)}{\Pb_{\be_0}(\bar{D}=\infty)}
$$
and
$$
\frac{d\hat{\Pb}_{\be}}{d\hat{\Pb}_{\be_0}}|_{\F_{\tau}}=M_{\tau}(\be,\be_0).
$$
A direct consequence is that
$$
\hat{\Eb}_{\be_0}[M_{\tau}(\be,\be_0)]=1\text{ and }\Eb_{\be_0}[M_{\tau}(\be,\be_0)]=\frac{\Pb_{\be}(\bar{D}=\infty)}{\Pb_{\be_0}(\bar{D}=\infty)}.
$$
Using the Girsanov's transform, we get the formula of the speed:
$$
v(\be)=\frac{\hat{\Eb}_{\be}X_{\tau}}{\hat{\Eb}_{\be}\tau}=\frac{\hat{\Eb}_{\be_0}[X_{\tau}M_{\tau}(\be,\be_0)]}{\hat{\Eb}_{\be_0}[\tau M_{\tau}(\be,\be_0)]}
$$
On the other hand,
\begin{align}\label{dhM}
\frac{\partial}{\partial\be}[M_{\tau}(\be,\be_0)]=\frac{\partial}{\partial\be}\left[\prod_{i=0}^{\tau-1}\left(\frac{1+\be\E_i1_{Y_i\notin}}{1+\be_0\E_i1_{Y_i\notin}}\right)\right]=\left[\sum_{i=0}^{\tau-1}\left(\frac{\E_i1_{Y_i\notin}}{1+\be\E_i1_{Y_i\notin}}\right)\right]M_{\tau}(\be,\be_0).
\end{align}
Set $V_{\tau}=\sum_{i=0}^{\tau-1}\left(\frac{\E_i1_{Y_i\notin}}{1+\be\E_i1_{Y_i\notin}}\right)$ then 
$$
\int_{\be_0}^{\be}{M_{\tau}(t,\be_0)V_{\tau}(t)}dt=\int_{\be_0}^{\be}\frac{\partial}{\partial t}{M_{\tau}(t,\be_0)}dt=M_{\tau}(\be,\be_0)-M_{\tau}(\be_0,\be_0)=M_{\tau}(\be,\be_0)-1.
$$
Therefore,
$$
v(\be)=\frac{\hat{\Eb}_{\be_0}\left[X_{\tau}\left(1+\int_{\be_0}^{\be}{M_{\tau}(t,\be_0)V_{\tau}(t)}dt\right)\right]}{\hat{\Eb}_{\be_0}\left[{\tau}\left(1+\int_{\be_0}^{\be}{M_{\tau}(t,\be_0)V_{\tau}(t)}dt\right)\right]}
$$
To prove the existence of the derivative, we apply the Fubini's theorem as follows:
\begin{thm}[Fubini's theorem]
Let $\mu,\nu$ be the $\si-$finite mesures. If either
 \begin{center}
$\int_{A}\left(\int_{B}|f(x,y)|\nu(dy)\right)\mu(dx)<+\infty$ or $\int_{B}\left(\int_{A}|f(x,y)|\mu(dx)\right)\nu(dy)<+\infty$
 \end{center} then $\int_{A\times B}|f(x,y)|(\mu\times\nu)(dxdy)<+\infty$ and
 \begin{center}
 $\int_{A\times B}f(x,y)(\mu\times\nu)(dxdy)=\int_{A}\left(\int_{B}f(x,y)\nu(dy)\right)\mu(dx)=\int_{B}(\int_{A}f(x,y)\mu(dx))\nu(dy)$$.$
\end{center}  
 \end{thm}
To apply the Fubini's theorem, let $\be\in(\be_0-\de,\be_0+\de)\subset(0,1)$ we observe that
\begin{align*}
&\int_{\be_0}^{\be}(\Eb_{\be_0}|X_{\tau}V_{\tau}M_{\tau}|)dt\leq\int_{\be_0}^{\be}\Eb_{\be_0}\left(\tau^2M_{\tau}\frac{1}{1-t}\right)dt\\
&\leq \frac{1}{1-\be_0-\de}\int_{\be_0}^{\be}[\Eb_t(\tau^2)]dt<+\infty.
\end{align*}
The last inequality above is implied since $\sup_{t\in(\be_0-\de,\be_0+\de)}\Pb_t(\tau>n)\leq Ce^{-n^{\al}}$ then $$\sup_{t\in(\be_0-\de,\be_0+\de)}\Eb_t(\tau^2)<+\infty.$$
It remains to prove that $\hat{\Eb}_{\be_0}(X_{\tau}V_{\tau}M_{\tau})$ is continuous in $\be$, this is true if let an interval $J=(a,b)\subset(0,1)$ we have that $\{(X_{\tau}V_{\tau}M_{\tau})\}_{\be\in J}$ is uniformly integrable.
Let $\be_1\in J$, observe that
\begin{itemize}
\item $|X_{\tau}V_{\tau}M_{\tau}|\leq C\tau^2M_{\tau}$ for $C=\frac{1}{b},$
\item $\lim_{\be\to\be_1}(\tau^2M_{\tau})(\be)=\tau^2M_{\tau}(\be_1),$
\item $\lim_{\be\to\be_1}\hat{\Eb}_{\be_0}[(\tau^2M_{\tau})(\be)]=\hat{\Eb}_{\be_0}[(\tau^2M_{\tau})(\be_1)].$
\end{itemize}
The last equality is implied from the fact that 
$$
\hat{\Eb}_{\be_0}[(\tau^2M_{\tau})(\be)]=\hat{\Eb}_{\be_0}\left[\int_{\be_1}^{\be}(\tau^2V_{\tau}M_{\tau})(t)dt\right]+\hat{\Eb}_{\be_0}[(\tau^2M_{\tau})(\be_1)]
$$
and
\begin{align*}
&\hat{\Eb}_{\be_0}\left[\int_{\be_1}^{\be}(\tau^2V_{\tau}M_{\tau})(t)dt\right]=\int_{\be_1}^{\be}\hat{\Eb}_{\be_0}\left[(\tau^2V_{\tau}M_{\tau})(t)\right]dt
\leq C\int_{\be_1}^{\be}\hat{\Eb}_{\be_0}\left[(\tau^3M_{\tau})(t)\right]dt\\
&=C\int_{\be_1}^{\be}\hat{\Eb}_t(\tau^3)dt\leq C\sup_{J}\hat{\Eb}_t(\tau^3)(\be-\be_1)\to 0\text{ as }\be\to\be_1, \sup_{J}\hat{\Eb}_t(\tau^3)<\infty \text{ since Lemma \ref{estimatau}.}
\end{align*}
From the observation above we imply that $\{\tau^2M_{\tau}\}_{\be\in J}(\be)$ and also $\{X_{\tau}V_{\tau}M_{\tau}\}_{\be\in J}(\be)$ is uniformly integrable then $\hat{\Eb}_{\be_0}[(X_{\tau}V_{\tau}M_{\tau})(\be)]$ is continuous.
\end{proof}
We rewrite the formula of the speed:
$$
v(\be)=\frac{\hat{\Eb}_{\be_0}\left[X_{\tau}\left(1+\int_{\be_0}^{\be}{M_{\tau}(t,\be_0)V_{\tau}(t)}dt\right)\right]}{\hat{\Eb}_{\be_0}\left[{\tau}\left(1+\int_{\be_0}^{\be}{M_{\tau}(t,\be_0)V_{\tau}(t)}dt\right)\right]}
=\frac{\hat{\Eb}_{\be_0}X_{\tau}+\int_{\be_0}^{\be}\left[{\Eb}_{\be_0}\left({X_{\tau}M_{\tau}(t,\be_0)V_{\tau}(t)}dt\right)\right]}{\hat{\Eb}_{\be_0}{\tau}+\int_{\be_0}^{\be}\left[{\Eb}_{\be_0}\left({\tau M_{\tau}(t,\be_0)V_{\tau}(t)}dt\right)\right]}
$$
Set $A:=\hat{\Eb}_{\be_0}X_{\tau}+\int_{\be_0}^{\be}\left[{\Eb}_{\be_0}\left({X_{\tau}M_{\tau}(t,\be_0)V_{\tau}(t)}dt\right)\right]$ and $B:=\hat{\Eb}_{\be_0}{\tau}+\int_{\be_0}^{\be}\left[{\Eb}_{\be_0}\left({\tau M_{\tau}(t,\be_0)V_{\tau}(t)}dt\right)\right]$
Taking the derivative we obtain:
$$
\frac{\partial v}{\partial\be}(\be)=\frac{{\Eb}_{\be_0}\left({X_{\tau}M_{\tau}(\be,\be_0)V_{\tau}(\be)}\right)B-{\Eb}_{\be_0}\left({{\tau}M_{\tau}(\be,\be_0)V_{\tau}(\be)}\right)A}{B^2}.
$$
As $\be=\be_0$,
$$
\frac{\partial v}{\partial\be}(\be_0)=\frac{{\Eb}_{\be_0}\left({X_{\tau}V_{\tau}(\be_0)}\right)\hat{\Eb}_{\be_0}{\tau}-{\Eb}_{\be_0}\left({{\tau}V_{\tau}(\be_0)}\right)\hat{\Eb}_{\be_0}X_{\tau}}{(\hat{\Eb}_{\be_0}{\tau})^2}.
$$
Therefore, for all $\be\in(0,1)$ we have
\begin{equation}\label{deri}
\frac{\partial v}{\partial\be}(\be)=\frac{{\Eb}_{\be}\left({X_{\tau}V_{\tau}}\right)\hat{\Eb}_{\be}{\tau}-{\Eb}_{\be}\left({{\tau}V_{\tau}}\right)\hat{\Eb}_{\be}X_{\tau}}{(\hat{\Eb}_{\be}{\tau})^2}=\frac{{\Eb}_{\be}\left[({X_{\tau}-v\tau)V_{\tau}}\right]\hat{\Eb}_{\be}{\tau}}{(\hat{\Eb}_{\be}{\tau})}.
\end{equation}
From \ref{limder} and \ref{deri} we get that
$$
\lim_{n\to+\infty}\frac{\partial v_n}{\partial\be}(\be)=\frac{\partial v}{\partial\be}(\be)=\frac{{\Eb}_{\be}\left[({X_{\tau}-v\tau)V_{\tau}}\right]\hat{\Eb}_{\be}{\tau}}{(\hat{\Eb}_{\be}{\tau})}.
$$
We have proved the differentiability of order $1$ of the speed $v_n(\be)$ and $v(\be)$. To prove the infinite differentiability of the speed we need the following lemma:
\begin{lem}\label{bddh}
Let $I=(a,b)$ be an open interval of $\Rb$, $(\Om,\F,\Pb)$ be a probability space and $H(x,\om)$ be a mapping
\begin{align*}
H: I\times \Om&\to \Rb\\
(x,\om)&\mapsto H(x,\om)
\end{align*}
such that for every $x\in I$, $H(x,\om)$ is a random variable, and for every $\om\in\Om$, $H(x,\om)$ is a smooth function on $I$. Moreover, suppose that for every $n\geq 0$
$$\sup_{x\in I}\Eb\left(\left|\frac{\partial^n H}{\partial x^n}(x,\om)\right|\right)<+\infty.$$
Then $\Eb[H(x,\om)]$ is a smooth function and for every $k\geq 1$
$$
\frac{\partial^k}{\partial x^k}[\Eb(H(x,\om))]=\Eb\left(\left|\frac{\partial^k H}{\partial x^k}(x,\om)\right|\right).
$$

\end{lem}
This lemma can be proved by the induction in $k$ and using Fubini's Theorem.

From \eqref{dhM},
\begin{equation}
\label{estM}
\frac{\partial^{n+1}}{\partial\be^{n+1}}[M_{\tau}(\be,\be_0)]=\frac{\partial^{n}}{\partial\be^{n}}[V_{\tau}(\be)M_{\tau}(\be,\be_0)]=\sum_{k=0}^{n}C^k_n\frac{\partial^{k}}{\partial\be^{k}}[V_{\tau}(\be)]\frac{\partial^{n-k}}{\partial\be^{n-k}}[M_{\tau}(\be,\be_0)].
\end{equation}

We have, for all $k\geq 0$ 
\begin{equation}\label{estV}
\sup_{\be\in I}\left|\frac{\partial^{k}}{\partial\be^{k}}[V_{\tau}(\be)]\right|=\sup_{\be\in I}\left|(-1)^kk!\sum_{i=0}^{\tau-1}\left(\frac{(\E_i1_{Y_i\notin})^{k+1}}{(1+\be\E_i1_{Y_i\notin})^{k+1}}\right)\right|\leq \frac{k!\tau}{(1-\be_0-\de)^{k+1}}.
\end{equation}

We will prove by induction in $n$ that 
\begin{equation}\label{estMI}
\left|\frac{\partial^{n}}{\partial\be^{n}}[M_{\tau}(\be,\be_0)]\right|\leq \sum_{k=0}^nc_{kn}{\tau}^kM_{\tau}(\be,\be_0)
\end{equation}
where $c_{kn}$ are non-negative constants depending only on $n,\be_0,\de.$ For $n=0$, it is true with $c_{00}=1.$ Suppose that it is true up to $n\geq 0.$ For $n+1$ then by induction supposition combining \eqref{estM}, \eqref{estV} then 
$$
\left|\frac{\partial^{n+1}}{\partial\be^{n+1}}[M_{\tau}(\be,\be_0)]\right|\leq\sum_{k=0}^{n}C^k_n\frac{k!\tau}{(1-\be_0-\de)^{k+1}}\sum_{i=0}^{n-k}c_{i,n-k}{\tau}^iM_{\tau}(\be,\be_0)=\sum_{i=0}^{n+1}c_{i,n+1}{\tau}^iM_{\tau}(\be,\be_0)
$$
where $c_{(i+1)(n+1)}=\sum_{k=0}^{n}C^k\frac{k!}{(1-\be_0-\de)^{k+1}}_n c_{i,n-k}$ for $i=1$ to $n$ and $c_{0,n+1}=0.$ This proved \eqref{estMI}.

On $I=(\be_0-\de,\be_0+\de)$ then 
\begin{align*}
\sup_{\be\in I}\hat{\Eb}_{\be_0}\left[\left|\frac{\partial^n}{\partial\be^n}[X_{\tau}M_{\tau}(\be,\be_0)]\right|\right]=\sup_{\be\in I}\hat{\Eb}_{\be_0}\left[\left|X_{\tau}\frac{\partial^n}{\partial\be^n}[M_{\tau}(\be,\be_0)]\right|\right].
\end{align*}
Since $|X_{\tau}|\leq \tau$ then 
\begin{align}\label{dhbc}
\sup_{\be\in I}\hat{\Eb}_{\be_0}\left[\left|\frac{\partial^n}{\partial\be^n}[X_{\tau}M_{\tau}(\be,\be_0)]\right|\right]\leq \sup_{\be\in I}\sum_{k=0}^nc_{kn}\hat{\Eb}_{\be}\left[\tau^{k+1}\right]<+\infty.
\end{align}
The last inequality is implied by
$\sup_{t\in(\be_0-\de,\be_0+\de)}\Pb_t(\tau>n)\leq Ce^{-n^{\al}}$ then $$\sup_{t\in(\be_0-\de,\be_0+\de)}\Eb_t(\tau^n)<+\infty \text{ for all }n\geq 1.$$
Combining \eqref{dhbc} with Lemma \ref{bddh} we get the smoothness of $\hat{\Eb}_{\be_0}[X_{\tau}M_{\tau}(\be,\be_0)]$ and similarly for $\hat{\Eb}_{\be_0}[{\tau}M_{\tau}(\be,\be_0)]$. This implies the smoothness of the speed $v(\be)$. Lemma \ref{bddh} implies similarly the smoothness of $v_n(\be)$.
We have proved that for $k=1$ and $\be>0$, there exists the limit $\lim_{n\to+\infty}\frac{\partial^k v_n}{\partial\be^k}(\be)$ such that
$$
\lim_{n\to+\infty}\frac{\partial^k v_n}{\partial\be^k}(\be)=\frac{\partial^k v}{\partial\be^k}(\be).
$$
For $k>1$, it is very complicated for computation so we leave it for reader.
If we write the speed in the form
$
v(\be)=\frac{\be}{d}.\frac{\hat{\Eb}_{\be}N_{\tau}}{\hat{\Eb}_{\be}\tau}
$
, we can get the formula \eqref{deriv} of the derivative.
We proved the first point of Theorem \ref{smallD}.
\subsection{The existence of the derivative at the critical point $0$}
Denote the event $\{Y_n\notin\{Y_{n-1},Y_{n-2},...,Y_{n-k}\}\}$ by $\{Y_n\notin^k\}$ with the convention that if $n\leq k$ then two events $\{Y_n\notin\{Y_{n-1},Y_{n-2},...,Y_{n-k}\}\}$ and $\{Y_n\notin\}$ agree. 
Set $N^{(k)}_n:=1_{Y_0\notin^k}+1_{Y_1\notin^k}+...+1_{Y_{n-1}\notin^k}.$
We need the following lemma:
\begin{lem}
There exists a non negative constant $N^{(k)}(\be)$ such that $\Pb_{\be}-$a.s.
$$
\lim_{n\to\infty}\frac{N^{(k)}_n}{n}=N^{(k)}(\be).
$$
\end{lem}
\begin{proof}
The above result is easy to verify by considering two following cases:

If $\be>0$ then there exists a sequence of renewal times $\{\tau_n\}$ and the sequence $\{N^{(k)}_{\tau_n}-N^{(k)}_{\tau_{n-1}}\}$ is independent. It is similar as the law of large number for $\frac{X_n}{n}$ we also have $$\lim_{n\to\infty}\frac{N^{(k)}_n}{n}=N^{(k)}(\be).$$

If $\be=0$ then $((\Zb^d)^{\Nb},\th,\Pb_0)$ is a system ergodic where $Y_n\circ\th =Y_{n+1}-Y_1.$
For $n\geq k$ then $\{Y_n\notin^k\}=\{Y_k\circ\th^{n-k}\notin^k\}$, therefore
\begin{align*}
\lim_{n\to\infty}\frac{N^{(k)}_n}{n}&=\lim_{n\to\infty}\frac{\sum_{j=0}^{k-1}1_{Y_j\notin}+\sum_{i=0}^{n-k-1}1_{Y_k\circ\th^{i}\notin^k}}{n}\\
&=\lim_{n\to\infty}\frac{\sum_{i=0}^{n-k-1}1_{Y_k\circ\th^{i}\notin^k}}{n-k}.\frac{n-k}{n}
=\Pb_0(Y_k\notin).
\end{align*}
\end{proof}

Observe that when $k$ increases then $1_{Y_n\notin^k}$ decreases and $\lim_{k\to\infty}1_{Y_n\notin^k}=1_{Y_n\notin}.$ Set $N_n=1_{Y_0\notin}+1_{Y_1\notin}+...+1_{Y_{n-1}\notin}$ then $\Pb_{\be}-$a.s. we have $N(\be):=\lim_{n\to\infty}\frac{N_n}{n}.$ We will prove a result as follows:
\begin{lem}
When $k$ tend to infinity, $N^{(k)}(\be)$ decreases to $N(\be)$ and uniformly for $d\geq 4.$
\end{lem}
\begin{proof}
Indeed, we have
\begin{align}\label{estuniv}
&|\Pb_{\be}(Y_n\notin^k)-\Pb_{\be}(Y_n\notin)|=\Pb_{\be}[(Y_n\notin^k)\cap[(Y_n=Y_{n-k-1})\cup(Y_n=Y_{n-k-2})...\cup(Y_n=Y_0)]]\notag\\
&\leq \sum_{j=0}^{n-k-1}\Pb_{\be}(Y_n=Y_j)\leq\sum_{j=0}^{n-k-1}\Pb_{\be}(Z_n=Z_j)=\Pb(Z_{k+1}=0)+\Pb(Z_{k+2}=0)+...+\Pb(Z_{n}=0)\notag\\
&\leq \sum_{j=k+1}^{\infty}\Pb(Z_{j}=0).
\end{align}
From the sequences $(Z_k)_{k\in\Zb}$, $(\et_k)_{k\in\Zb}$ , $(\xi_k)_{k\in\Zb}$, $(\ze_k)_{k\in\Zb}$, 
we can construct the ERW $(Y_n)_{n\geq 0}$ just as in the first construction. We also define  the sequence 
$(\tilde{Z}_k)_{k\in\Zb}$ as the sequence of "moves" of $Z$. More precisely, $(\tilde{Z}_k)_{k\in\Zb}$ is the 
unique sequence such that:
\begin{equation}
\label{lienZ-tildeZ}\forall n\geq 0,\quad Z_n= \begin{cases}
		\tilde{Z}_{\sum_{i=0}^{n-1}(1-\et_i)} & \mbox{ if } n \geq 0 \, ; \\
		\tilde{Z}_{\sum_{i=n}^{-1}(1-\et_i)} & \mbox{ if } n < 0 \, .
		\end{cases}
\end{equation}
$\et_k:=1_{Z_k=Z_{k+1}}$
Set $\et_i=1_{Z_i=Z_{i+1}}$ and $U=\sum_{i=0}^{n-1}\et_i$. We define $(\tilde{Z}_k)_{k\in\Zb}$ as the sequence of "moves" of $Z$, see \eqref{lienZ-tildeZ}.
Using \cite{Spi01}, page 75, we obtain $\Pb(\tilde{Z}_i=0)\sim i^{_-\frac{d-1}{2}}.$
We have 
\begin{align*}
&\Pb(Z_n=0)=\sum_{k=0}^{n}\Pb(\tilde{Z}_k=0).C_n^{n-k}\left(\frac{1}{d}\right)^{n-k}\left(\frac{d-1}{d}\right)^k\\
&=\sum_{k=0}^{\frac{n}{2d}}\Pb(\tilde{Z}_k=0).C_n^{n-k}\left(\frac{1}{d}\right)^{n-k}\left(\frac{d-1}{d}\right)^k+\sum_{k=\frac{n}{2d}+1}^{n}\Pb(\tilde{Z}_k=0).C_n^{n-k}\left(\frac{1}{d}\right)^{n-k}\left(\frac{d-1}{d}\right)^k.
\end{align*}
We estimate the first term:
\begin{align*}
&\sum_{k=0}^{\frac{n}{2d}}\Pb(\tilde{Z}_k=0).C_n^{n-k}\left(\frac{1}{d}\right)^{n-k}\left(\frac{d-1}{d}\right)^k\leq \sum_{k=0}^{\frac{n}{2d}}C_n^{n-k}\left(\frac{1}{d}\right)^{n-k}\left(\frac{d-1}{d}\right)^k\\
&=\Pb\left(\frac{U-n.\frac{1}{d}}{\sqrt{n}}\leq\frac{-\sqrt{n}}{2d}\right)
\sim\int_{-\infty}^{-\frac{\sqrt{n}}{2d}}\frac{1}{\sqrt{2\pi}}e^{-\frac{x^2}{2}}dx
\text{ converging to }0 \text{ when } n\to\infty.
\end{align*}
On the other hand
\begin{align*}
\sum_{k=\frac{n}{2d}+1}^{n}\Pb(\tilde{Z}_k=0).C_n^{n-k}\left(\frac{1}{d}\right)^{n-k}\left(\frac{d-1}{d}\right)^k\sim C.n^{-\frac{d-1}{2}}
\end{align*}
since 
$$
\frac{1}{2}\leq\sum_{k=\frac{n}{2d}+1}^{n}C_n^{n-k}\left(\frac{1}{d}\right)^{n-k}\left(\frac{d-1}{d}\right)^k\leq 1
$$
and
$$
\Pb(\tilde{Z}_k=0)\sim C'.n^{-\frac{d-1}{2}}\text{ for }\frac{n}{2d}\leq k\leq n.
$$
From \eqref{estuniv} we get,
$$
|\Pb_{\be}(Y_n\notin^k)-\Pb_{\be}(Y_n\notin)|\leq C\sum_{j=k+1}^{\infty}j^{-\frac{d-1}{2}}.
$$
It implies
\begin{align*}
\left|\Eb_{\be}\left(\frac{N^{(k)}_n}{n}\right)-\Eb_{\be}\left(\frac{N_n}{n}\right)\right|&\leq\frac{1}{n}\sum_{i=0}^{n-1}|\Pb_{\be}(Y_i\notin^k)-\Pb_{\be}(Y_i\notin)|\\
&\leq C\sum_{j=k+1}^{\infty}j^{-\frac{d-1}{2}}.
\end{align*}
Let $n$ converge to infinity then $|N^{k}(\be)-N(\be)|\leq C\sum_{j=k+1}^{\infty}j^{-\frac{d-1}{2}}$. If $d\geq 4$ then $N^{k}$ converges uniformly to $N$ in $\be.$
\end{proof}
Now, we return to the proof of the point 2 of Theorem \ref{smallD}. By $\frac{v(\be)}{\be}=N(\be)$, to prove the existence of the derivative at $0$ we need to prove that $N(\be)$ is continuous at $0.$
It is known that $N^{k}$ converges uniformly to $N$ in $\be$ for $d\geq 4$, then there is just one thing left is to show that $N^{k}(\be)$ is continuous at $0.$
Indeed, 
\begin{align*}
&|\Pb_{\be}(Y_n\notin^k)-\Pb_{0}(Y_n\notin^k)|\\
&=\left|\Eb_0\left[1_{Y_n\notin^k}\prod_{j=0}^{n-1}(1+\E_j\be 1_{Y_j\notin})\right]-\Eb_0\left[1_{Y_n\notin^k}\prod_{j=0}^{n-k}(1+\E_j\be 1_{Y_j\notin})\right]\right|\\
&\leq \Eb_0\left[1_{Y_n\notin^k}\prod_{j=0}^{n-k}(1+\E_j\be 1_{Y_j\notin})\left|\prod_{j=n-k+1}^{n-1}(1+\E_j\be 1_{Y_j\notin})-1\right|\right]\\
&\leq [(1+\be)^{k-1}-1]\Eb_0\left[1_{Y_n\notin^k}\prod_{j=0}^{n-k}(1+\E_j\be 1_{Y_j\notin})\right]=[(1+\be)^{k-1}-1]\Pb_0(Y_n\notin^k)\\
&\leq [(1+\be)^{k-1}-1].
\end{align*}
Hence, $\left|\Eb_{\be}\left(\frac{N^{(k)}_n}{n}\right)-\Eb_{0}\left(\frac{N^{(k)}_n}{n}\right)\right|\leq [(1+\be)^{k-1}-1] $ and $|N^k(\be)-N^k(0)|\leq[(1+\be)^{k-1}-1].$
This implies that $N^k(\be)$ is continuous at $0.$
Therefore, for $d\geq 4$ then
\begin{itemize}
\item $N^k(\be)$ converges uniformly to $N(\be)$ when $k\to\infty,$
\item $N^k(\be)$ is continuous for every $k>1.$
\end{itemize}
We deduce that $N(\be)$ is continuous at $0$, it means that $$\lim_{\be\to 0}\frac{v(\be)}{\be}=\frac{1}{d} N(0)=\frac{1}{d}\lim_{n\to\infty}\frac{R_n}{n}=\frac{1}{d}R(0).$$
Notice that $R_n=\Eb_0(N_n)$\\
For $d=2$, $R(0)=N(0)=0$, see in \cite{LGR91}. Let $\si>0$, on one hand, since $N^k(0)$ decreases to $N(0)$ when $k\to\infty$ then there exists $k_0$ such that $N^k(0)<\si$ for all $k\geq k_0.$ On the other hand, $N^{k_0}(\be)$ is continuous at $0$ then there exists $\be_0>0$ such that $|N^{k_0}(\be)-N^{k_0}(0)|<\si$ for all $\be<\be_0.$ Since $N(\be)\leq N^{k_0}(\be)$, $N(\be)\leq N^{k_0}(\be)\leq N^{k_0}(0)+\si<2\si$ for all $\be<\be_0.$ This implies that $\lim_{\be\to 0}N(\be)=N(0)=0.$ 
Therefore $\lim_{\be\to 0}\frac{v(\be)}{\be}=0.$

For $d=3$, because of $N(\be)\leq N^k(\be)$ for all $k>1$ then 
$$
\limsup_{\be\to 0}N(\be)\leq\lim_{k\to\infty}\limsup_{\be\to 0}N^k(\be)
=\lim_{k\to\infty}N^k(0)=\frac{1}{d}R(0).
$$

\section{The proof of Theorem \ref{monotonesmall}}
\subsection{The monotonicity of the range of the simple random walk}

Firstly, for the range of the simple random walk, we have a known result as follows (see \cite{Spi01}, \cite{DE51}):
$$R(\be)=\Pb_{\infty,\be}[Y_0\notin Y_{[1,\infty)}]$$
Then, we obtain:
\begin{align}\label{step1}
1-R(\be)&=\Pb_{\infty,\be}[\exists n>0 \text{ such that }Y_n=Y_0=0]\notag\\
&=\Pb_{\infty,\be}\left\{\bigcup_{k=1}^{\infty}[Y_{2k}=0\text{ and }0\notin Y_{[1,2k)}]\right\}\notag\\
&=\sum_{k=1}^{\infty}\Pb_{\infty,\be}\left\{[Y_{2k}=0\text{ and }0\notin Y_{[1,2k)}]\right\}.
\end{align}
On the other hand, we see that the trajectories with $2k$ steps $\{y_0=0,y_1,y_2,...,y_{2k-1},y_{2k}=0\}$ start from the origin and return at the origin at the time $2k$ whose number of jumps to the left equal to the number of jumps to right that we denote equal to $a_1.$
Therefore
\begin{align}\label{step2}
&\Pb_{\infty,\be}\left\{[Y_{2k}=0\text{ et }0\notin Y_{[1,2k)}]\right\}\notag\\
&=\sum_{\{y_0=0,y_1,...,y_{2k}=0\}}\left(\frac{1+\be}{2d}\right)^{a_1}\left(\frac{1-\be}{2d}\right)^{a_1}\left(\frac{1}{2d}\right)^{a_2}\text{ where }2a_1+a_2=2k\notag\\
&=\sum\left(\frac{1-\be^2}{(2d)^2}\right)^{a_1}\left(\frac{1}{2d}\right)^{a_2}.
\end{align}
 From \eqref{step1} and \eqref{step2}, we imply that $1-R(\be)$ is decreasing then $R(\be)$ is increasing in $\be$.
 \subsection{The monotonicity of the speed of excited random walk with several identical cookies}
To prove the monotonicity of the speed for $m$-ERW we need the following lemma:
\begin{lem}\label{lemabasic}
Let $J$ be an interval of $\Rb$ and $\{X_n(\be)\}_{\be \in J, n \geq 1}$, $\{X(\be)\}_{\be\in J}$
the families of positive random variables. Under supposition that 
\begin{enumerate}
\item for every $n$, $\{ X_n(\be)\}_{\be\in J}$ is  uniformly integrable,
\item $\{X(\be)\}_{\be\in J}$ is uniformly integrable,
\item $X_n(\be)$ converges in probability to $X(\be)$, uniformly in $\be$: for every $\ep >0$, 
$$ \lim_{n \to +\infty} \sup_{\be \in J} \Pb ( \left| X_n(\be) - X(\be) \right| > \ep) =0 \, . 
$$ 
\end{enumerate}
Then, $\lim_{n \rightarrow + \infty}  \sup_{\be \in J} \left| \Eb(X_n(\be))  -\Eb(X(\be)) \right| = 0$
 if and only if $\{X_n(\be)\}_{n\in\Nb,\be\in J}$ is uniformly integrable.
\end{lem}
This lemma is proved in the end of \cite{Pha15}.

\subsubsection{Coupling of random walks}\label{coup}

The method coupling is usually used to prove the recurrence or the monotonicity of random walks. For example, in \cite{BW03}, the authors coupled an excited random (ERW) walk with a simple symmetric random walk to prove the recurrence of ERW. In \cite{BFS11}, this method is used to prove the monotonicity by coupling tree random walks, two biased random walks with bias respectively $\be$ and $\be+\ep$ on the Galton-Watson tree, a simple random walk with bias $\be$ on $\Zb.$ In our paper, to prove the monotonicity of the speed on $[\be_0,1]$ we need to couple the random walk $\bar{Y}$ of bias $\be_0$ with the $m$-ERW $Y$ of bias $\be$ where $\be\geq \be_0$.

Let a probability space:
$$
\Om=(\Zb^{d-1})^{\Zb}\times(\{0,1\}^3)^{\Zb}
$$
endowed the  product $\si$-algebra $\F$ and the product probability 
$$
\Pb=q^{\Zb}\otimes p^{\Zb}
$$
where
\begin{itemize}
\item $q$ is the law of the increments of $Z$ i.e. for $e\in\Zb^{d-1}$ then $q(e)=\frac{1}{2d}$ if $|e|=1$ and $q(e)=0$ if $e=0$,
\item For $(x,y,z)\in \{0,1\}^3$, let $ p_{xyz}=p[(x,y,z)]$ such that:
\begin{gather} p_{111}=\frac{1}{2},p_{011}=\frac{\be_0}{2},p_{001}=\frac{\be-\be_0}{2},p_{000}=\frac{1-\be}{2} \text{ and for the other cases} p_{xyz}=0.\label{dis}
\end{gather} 
\end{itemize}Now, we take $\om=(u,v,l,h)\in \Om$ with $u\in{(\Zb^{d-1})}^{\Zb}$, $(v,l,h)\in(\{0,1\}^3)^{\Zb}$. 
Let $(\theta_n)_{n\in\Zb}$ be the canonical shift on $\Om$ and $(I_n, \xi_n, \ze_n, \bar{\ze}_n)_{n\in\Zb}$ be the canonical process: 
$$
I_n(\om)=u_n \in \Zb^{d-1} \, , \, \,
\xi_n(\om)=v_n \in  \{0,1\} \, , \,\,  \bar{\ze}_n(\om)=l_n \in  \{0,1\} \, , \,\, 
\ze_n(\om)=h_n \in  \{0,1\} \, .
$$ 
Define $Z$ as follows:
\begin{equation}
Z_n=
\begin{cases}
I_1+I_2+...+I_n, &n\geq 1,\\
0, &n=0,\\
-(I_{n+1}+...+I_0), &n\leq -1
\end{cases}
\end{equation}
then $(Z_n)_{n\in\Zb}$ is a symmetric simple random walk on $\Zb^{d-1}$ such that $Z_0=0,\, a.s.$ and it does not move at every site with probability $\frac{1}{d}$, it jumps from a site to every next site with probability $\frac{1}{2d}$.
Let $\et_i:=1_{Z_i=Z_{i+1}}$. By the construction above
$(\et_i)_{i\geq 0}$,$(\xi_i)_{i\geq 0}$, $(\ze_i)_{i\geq 0}$ and  $(\bar{\ze}_i)_{i\geq 0}$ such that the random vectors of the sequence $((\et_i,\xi_i,\ze_i,\bar{\ze}_i))_{i\in\Zb}$ are independent, two sequences $(\et_i)_{i\in\Zb}$ and $((\xi_i,\ze_i,\bar{\ze}_i))_{i\in\Zb}$ are independent. On the other hand, the vector $(\et_i,\xi_i,\ze_i,\bar{\ze}_i)$ satisfies:
\begin{itemize}
\item$
\et_i\sim Ber\left(\frac{1}{d}\right),\xi_i\sim Ber\left(\frac{1}{2}\right),\bar{\ze}_i\sim Ber\left(\frac{1+\be_0}{2}\right),\ze_i\sim Ber\left(\frac{1+\be}{2}\right)\text{ and }\xi_i\leq \bar{\ze}_i \leq \ze_i.
$
\item $\Pb(\xi_i=x,\bar{\ze}_i=y,\ze_i=z)=p_{xyz}$ where $x,y,z\in\{0,1\}$.
\end{itemize}
 Let $Z_{(-\infty,n)}=\{Z_m, m<n\}$. Denote by $\{Z_n\notin\}$ the event $\{Z_n\notin Z_{(-\infty,n)}\}$ and we say that $Z_n$ is new. The complement of $\{Z_n\notin\}$  is denoted by $\{Z_n\in\}$ and we say $Z_n$ is old. We then define two random walks $(Y_n)_{n\geq 0}$ and $(\bar{Y}_n)$ as follows: Let $(Y_n)_{n\geq 0}$ be a random walk such that $Y_0=0,\,a.s.$, the vertical component $(Y_n\cdot e_2, Y_n\cdot e_3,...,Y_n\cdot e_d)=Z_n$. Let $\{Y_n\notin^m_{\leq}\}=[\#\{0\leq i\leq n, Y_i=Y_n\}\leq m]$ and the complement is denoted by $\{Y_n\in^m_{\leq}\}$. The horizontal component $X_n=Y_n\cdot e_1$ satisfies:
 \begin{itemize}
 \item  On the event $\{Y_n\notin^m_{\leq}\}$  then $\bar{\E}_n= X_{n+1}-X_n=(2\ze_n-1)1_{Z_n= Z_{n+1}}$.
 \item On the event $\{Y_n\in^m_{\leq}\}$, then $\E_n=X_{n+1}-X_n=(2\xi_n-1)1_{Z_n= Z_{n+1}}.$
  \end{itemize}
With the construction above, $(Y_n)$ is a $m-$ERW with bias $\be$.
Now, we will define the random walk $(\bar{Y}_n)$. We set $\bar{Y}_0=0$ $ a.s.$ The vertical component is $(\bar{Y}_n\cdot e_2,\bar{Y}_n\cdot e_3,...,\bar{Y}_n\cdot e_d)=Z_n,\, n\geq 0.$
The horizontal component is defined as follows:
\begin{itemize}
\item If $Z_n$ is new, set $\bar{\E}_n=\bar{X}_{n+1}-\bar{X}_n=(2\ze_n-1)1_{Z_n= Z_{n+1}}$,
\item If $Z_n$ is old, set $\E_n=\bar{X}_{n+1}-\bar{X}_n=(2\xi_n-1)1_{Z_n= Z_{n+1}}$.
\end{itemize}
Note that $\bar{Y}$ is a random walk with stationary increments. We call this random walk by SIRW.
The coupling above implies that $\bar{X}_{n+1}-\bar{X}_n\leq X_{n+1}-X_n$. Hence if $(\tau_n)_{n\geq 1}$ are renewal times of $\bar{Y}$ then they are also renewal times of $Y$.
We will use this property to prove the monotonicity of the speed of $m-$ERW when $m$ is large enough.
 Let $\D(\om)=\{n\in\Zb, X_{(-\infty,n-1]}(\om)<X_n(\om)\leq X_{[n,+\infty)}(\om)\}$ be the set of renewal times and the stationary point process $N(\om,dk)=\sum_{n\in\Zb}\de_n(dk)1_{n\in\D(\om)}$. We consider $$W=\{\om\in\Om, N(\om,(-\infty,0])=N(\om,[0,+\infty))=\infty\}.$$ Let $\{\tau_n\}$ be the sequence of renewal times of the walk $\{\bar{Y}_n\}$ such that $-\infty<...<\tau_{-2}<\tau_{-1}<\tau_0\leq 0<\tau_1<\tau_2<...<+\infty$.
By the construction, $\{\bar{Y}_n\}$ satisfies:
$$
\bar{Y}_{n+1}-\bar{Y}_n=\bar{Y}_1\circ\theta_n
$$
Hence, combining with the fact that $(\Om,\Pb,\theta)$ is ergodic then $\{\bar{Y}_n\}$ has the speed  $$\Pb-a.s.,\,\,\bar{v}(\be)=\lim_{n\to\infty}\frac{\bar{X}_n}{n}=\frac{\be}{d}\Pb(Z_0\notin).$$
For $d\geq 4$ we have $\bar{v}(\be)>0$, using the idea on the estimation of renewal times in \cite{MPRV12}, \cite{MP14} to get that
\begin{lem}\label{esttaubar}
Let $(\bar{Y}_n)_{n\in \Zb}$ be a SIRW on $\Zb^d$ for $d\geq 4$, with bias $\be\in(0,1]$ fixed and $(\tau_k,k\in \Zb)$ is the sequence of the renewal times respectively. Then, there exists $C,\al>0$ such that for every $n\in\Zb,$
$$\sup_{k\in\Zb}\Pb_{\be}[\tau_{k+1}-\tau_k|\G_0^{(k)}]\leq Ce^{-n^{\al}}\, a.s.$$
In particular, for every $k\in\Zb$ et $p\geq 1$ we have that $\tau_k<\infty\, p.s$ and $\Eb_{\be}[(\tau_{k+1}-\tau_k)^p]<\infty.$
\end{lem}
Set $\bar{D}_+=\{\bar{X}_n\geq 0\text{ for all }n\geq 0\}$, $\bar{D}_-=\{\bar{X}_n< 0\text{ for all }n< 0\}$, $D=\bar{D}_+\cap\bar{D}_-$ and $\hat{\th}=\th_{\tau}$.
It is similar as for simple random walk with drif $\be$, we believe that:
$\Pb(\bar{D}_+)={c_1(\be)}{\be^2}, \Pb(\bar{D}_-)={c_2(\be)}{\be^2}$ where for some positive constants $c_0,c_3 $ then $c_0\leq c_1(\be),c_2(\be)\leq c_3$. In fact, we do not use these properties for the proof of Theorem \ref{monotonesmall} then we do not prove them.
To prove Theorem \ref{monotonesmall}, we need one more lemma as follows:
\begin{lem} \label{lemmasta}
$\Pb(D)>0$ and $\Pb(W)=1$. Under $\hat{\Pb}(.)=\Pb(.|\bar{D}_+,\bar{D}_-
)$ i.e $\tau_0=0$ the sequence $\{\tau_{n+1}-\tau_n\}_{n\in\Zb}$ is stationary. Morever, the triples $(\Om,\Pb,\th)$ and $(\Om,\hat{\Pb},\hat{\th})$ are ergodic systems.
\end{lem}
\begin{proof}
The random walk $\bar{Y}$ has these speeds: 
$$
\bar{v}(\be)=\lim_{n\to+\infty}\frac{\bar{Y}_n.e_1}{n}=\frac{\be}{d}.\Pb(Z_0\notin)>0\text{ and }
\bar{v}_-(\be)=\lim_{n\to-\infty}\frac{\bar{Y}_n.e_1}{n}=-\frac{\be}{d}.\Pb(Z_0\notin)<0.
$$
Using the idea in \cite{MPRV12} we get $\Pb(D)>0.$ Because $\bar{Y}$ has stationary increments, $\Pb(D)>0$ implies that $\Pb(W)=1$. 
Now, we prove the remain part of Lemma \ref{lemmasta}.
$(\Om,\Pb,\th)$ is the standard ergodic system. We will prove it also is true for $(\Om,\hat{\Pb},\hat{\th})$.   
First, we prove that $\hat{\Pb}$ is invariant under $\hat{\th}$.
Take any set $A\subset W$. Without loss of generality, suppose that $A\subset (0\in\D)$, then we have:
\allowdisplaybreaks
{
\begin{align*}
\hat{\th}\circ\hat{\Pb}(A)
&=\hat{\Pb}\left(\hat{\th}^{-1}A\right)=\frac{\Pb\left(\th^{-1}_{\tau_1}A,D\right)}{\Pb(D)}\\
&=\sum_{k\geq 1}\frac{\Pb\left(\th^{-1}_kA,\tau_1=k,D\right)}{\Pb(D)}\\
&=\sum_{k\geq 1}\frac{\Pb\left(A,\tau_{-1}=-k,D\right)}{\Pb(D)}\\
&=\frac{\Pb(A)}{\Pb(D)}=\hat{\Pb}(A).
\end{align*}
}
Next, we prove that for any set $A\subset W$ such that $\hat{\th}^{-1}A=A$ then $\hat{\Pb}(A)=0 \text{ or }1.$ 
 Indeed, set $\hat{\Om}:=(0\in\D)$ and $B:=A\cap\hat{\Om}\subset W $.  Note that 
 $\hat{\th}^{-1}(\hat{\Om})=W$ then
$\hat{\th}^{-1}A=\hat{\th}^{-1}B$. This in turn implies that $\hat{\th}^{-1}B \cap \hat{\Om}
=\hat{\th}^{-1}A \cap \hat{\Om}= A \cap \hat{\Om}= B$.

We will prove that $\th_1\left[\hat{\th}^{-1}B\right]=\hat{\th}^{-1}B$.
Using the ergodicity of $(\Om, \Pb, \theta)$, it follows that $\Pb\left(\hat{\th}^{-1}B\right)$ $=0\text{ or }1$, 
and 
$$\hat{\Pb}(A)=\hat{\Pb}(\hat{\th}^{-1}A)= \hat{\Pb}(\hat{\th}^{-1}B)
=\frac{\Pb\left(\hat{\th}^{-1}B \cap \hat{\Om}\right)}{\Pb(\hat{\Om})}=0\text{ or }1.$$
Therefore, to finish the proof we only need to prove that $\th_1\left[\hat{\th}^{-1}B\right]=\hat{\th}^{-1}B$. Firstly, we show that $\th_1\left[\hat{\th}^{-1}B\right]\subset\hat{\th}^{-1}B$. 
Let $x\in\hat{\th}^{-1}B$ then $\hat{\th}x\in B.$ If $\tau_1(x)>1$,
 we have $\hat{\th}(\th_1x)=\hat{\th}x\in B$. Hence $\th_1x\in\hat{\th}^{-1}B.$ 
 If $\tau_1(x)=1$ then $\th_1x=\hat{\th}x\in B
 =\hat{\th}^{-1}B\cap\hat{\Om}$. This implies $\th_1x\in\hat{\th}^{-1}B.$ It remains to prove that $\hat{\th}^{-1}B\subset\th_1\left[\hat{\th}^{-1}B\right]$. 
Take $x\in\hat{\th}^{-1}B$ then $x=\th_1\left(\th_{-1}x\right)$ and we will prove that 
$\th_{-1}x\in\hat{\th}^{-1}B\Leftrightarrow\hat{\th}\left(\th_{-1}x\right)\in B.$ 
If $x\in\hat{\Om}$, then $\hat{\th}\left(\th_{-1}x\right)=x\in\hat{\th}^{-1}B\cap\hat{\Om}=B.$
 If $x\notin\hat{\Om}$, then $\hat{\th}\left(\th_{-1}x\right)=\hat{\th}x\in B.$
Because that the sequence of renewal times of the random walk $\bar{Y}$ is also for $Y$ the $m$-ERW with bias $\be.$ So, the increments $\{Y_{[\tau_n,\tau_{n+1})}\}_{n\in\Zb}$ are disjoint. Hence
$$
X_{\tau_{k+1}}-X_{\tau_k}=X_{\tau_1}\circ\hat{\th}_k \, .
$$
$$
\bar{X}_{\tau_{k+1}}-\bar{X}_{\tau_k}=\bar{X}_{\tau_1}\circ\hat{\th}_k \, .
$$
From the equations above and using the ergodicity of $(\Om,\hat{\Pb},\hat{\th})$, we apply that $\hat{\Pb}-$a.s. there exist $\bar{v}(\be), v(m,\be)>0$ such that
$$
v(m,\be)=\lim_{n\to +\infty}\frac{X_n}{n}=\frac{\hat{\Eb}X_{\tau}}{\hat{\Eb}\tau}.
$$
Note that if $\frac{X_n}{n}$ converges $\hat{\Pb}-a.s.$ to $v(m,\be)$ then it is true for $\Pb-a.s.$ Indeed, there exists a set $\hat{A}\subset D$ such that $\hat{\Pb}(\hat{A})=1$ and for all $\om\in\hat{A}$, $\frac{X_n}{n}(\om)$ converges to $v.$ We suppose that there exists a subset $B\subset \Om$ such that $\hat{\th}B\subset \hat{A}^c$. If $\Pb(B)>0$, by $B=\bigcup_{k=1}^{\infty}(B,T=k)$ then there exists $k$ such that $\Pb(B,T=k)>0.$ This implies that $\Pb[\th_k(B,T=k)]=\Pb(B,T=k)>0$. On the other hand $[\th_k(B,T=k)]\subset (\hat{\th}B)$, so $\Pb(\hat{\th}B)>0$ and $\hat{\Pb}(\hat{A}^c)=\frac{{\Pb}(\hat{A}^c)}{\Pb(D)}>0.$ This is contradictory with the supposition that $\hat{\Pb}(\hat{A})=1.$ Therefore, $\Pb(B)=0$. Now, let $B=\hat{\th}^{-1}(\hat{A}^c)$ then $B$ satisfies that $\hat{\th}B\subset \hat{A}^c$. This implies that $\Pb(B)=0$ and $\Pb(\hat{\th}^{-1}(\hat{A}))=1.$ For all $\om\in\hat{\th}^{-1}(\hat{A})$ then $\frac{X_n}{n}(\hat{\th}\om)$ converges to $v$, so $\frac{X_n}{n}(\om)$ converges to $v$. It means that $\frac{X_n}{n}$ converges to $v$ almost surely under $\Pb.$
\end{proof}

\subsubsection{Girsanov's transform} 
The couple $(\bar{Y},Y)$ takes its values in the space $U={(\Zb^d)}^{\Zb}\times{(\Zb^d)}^{\Nb}$. Consider $U^*=\{(\bar{y}_i)_{i\in\Zb}\times(y_j)_{j\in\Nb},\bar{y}_0=y_0=0;\bar{\ep}_i,\ep_j\in\{-1,1\}\text{ for }i\in\Nb,\bar{z_i}=z_i\text{ and }\bar{\ep}_i=\ep_i\text{ if }y_i\in^m_{\leq}\}$. Denote $\Pb_{m,\be_0,\be}$ the law of the couple $(\bar{Y},Y)$ and $\hat{\Pb}_{m,\be_0,\be}(\cdot)=\Pb_{m,\be_0,\be}(\cdot|D)$. Let $y=(y_i)_{i\in \Nb}, \bar{y}=(\bar{y}_i)_{i\in\Zb}$ and $z=(z_i)_{i\in\Zb}=(\bar{y}_i\cdot e
_2,...,\bar{y}_i\cdot e
_d)_{i\in\Zb}=(y_i\cdot e
_2,...,y_i\cdot e
_d)_{i\in\Zb}$ then
\begin{align*}
&q_n(m,\be)(y,\bar{y}):=\Pb_{m,\be_0,\be}[\bar{Y}_{n+1}=\bar{y}_{n+1},Y_{n+1}=y_{n+1}|(Z_i=z_i)_{i<0},\bar{Y}_0=Y_0=0,...,\bar{Y}_n=\bar{y}_n,Y_n=y_n]\\
&=\frac{1}{d}\left[\frac{1+\be_0}{2}1_{z_n\notin}1_{\bar{\ep}_n=\ep_n=1}+\frac{\be-\be_0}{2}1_{z_n\notin}1_{\bar{\ep}_n=-1,\ep_n=1}+\frac{1}{2}1_{z_n\notin}1_{\bar{\ep}_n=\ep_n=-1}
+\frac{1}{2}1_{z_n\in,y_n\notin^m_{\leq}}1_{\bar{\ep}_n
=\ep_n=1}+\right.\\
&\frac{\be}{2}1_{z_n\in,y_n\notin^m_{\leq}}1_{\bar{\ep}_n=-1,\ep_n=1}\left.+\frac{1}{2}1_{z_n\in,y_n\notin^m_{\leq}}1_{\bar{\ep}_n=\ep_n=-1}+\frac{1}{2}1_{y_n\in^m_{\leq}}1_{\bar{\ep}_n=\ep_n=1}+\frac{1}{2}1_{y_n\in^m_{\leq}}1_{\bar{\ep}_n=\ep_n=-1}+\frac{1}{2}1_{\bar{\ep}_n=\ep_n=0}\right]\\
&=\frac{1}{d}\left[\frac{1+\be_0}{2}1_{z_n\notin}1_{\bar{\ep}_n=\ep_n=1}+\frac{\be-\be_0}{2}1_{z_n\notin}1_{\bar{\ep}_n=-1,\ep_n=1}+\frac{1}{2}1_{\bar{\ep}_n=\ep_n=-1}
+\frac{1}{2}1_{z_n\in}1_{\bar{\ep}_n
=\ep_n=1}+\right.\\
&\frac{\be}{2}1_{z_n\in,y_n\notin^m_{\leq}}1_{\bar{\ep}_n=-1,\ep_n=1}\left.+\frac{1}{2}1_{\bar{\ep}_n=\ep_n=0}\right].
\end{align*}
The computation above is explained as follows: for example in the case $\{z_n\notin,\bar{\ep}_n=\ep_n=1\}$ we have $\bar{\ze}_n=1,\ze_n=1$. Then probability equals $p_{.11}=p_{011}+p_{111}=\frac{\be_0}{2}+\frac{1}{2}$ since \eqref{dis}.

Moreover, the law of $\bar{Y}$ does not depend on $m,\be$ then
\begin{align*}
&\Pb_{m,\be_0,\be}[\bar{Y}_0=Y_0=0,...,\bar{Y}_n=\bar{y}_n,Y_n=y_n,\bar{Y}_{n+1}=\bar{y}_{n+1},...,\bar{Y}_{n+k}=\bar{y}_{n+k}|(Z_i=z_i)_{i<0}]\\
&=\Pb_{m,\be_0,\be}[\bar{Y}_0=Y_0=0,...,\bar{Y}_n=\bar{y}_n,Y_n=y_n|(Z_i=z_i)_{i<0}]\\
&\times
\Pb_{m,\be_0,\be}[\bar{Y}_{n+1}=\bar{y}_{n+1},...,\bar{Y}_{n+k}=\bar{y}_{n+k}|(Z_i=z_i)_{i<0},\bar{Y}_0=Y_0=0,...,\bar{Y}_n=\bar{y}_n,Y_n=y_n]\\
&=\Pb_{m,\be_0,\be}[\bar{Y}_0=Y_0=0,...,\bar{Y}_n=\bar{y}_n,Y_n=y_n|(Z_i=z_i)_{i<0}]\\
&\times
\Pb_{1,\be_0,\be_0}[\bar{Y}_{n+1}=\bar{y}_{n+1},...,\bar{Y}_{n+k}=\bar{y}_{n+k}|(Z_i=z_i)_{i<0},\bar{Y}_0=Y_0=0,...,\bar{Y}_n=\bar{y}_n,Y_n=y_n].
\end{align*}
Therefore,
\begin{align*}
&\frac{\Pb_{m,\be_0,\be}[\bar{Y}_0=Y_0=0,...,\bar{Y}_n=\bar{y}_n,Y_n=y_n,\bar{Y}_{n+1}=\bar{y}_{n+1},...,\bar{Y}_{n+k}=\bar{y}_{n+k}|(Z_i=z_i)_{i<0}]}{\Pb_{1,\be_0,\be_0}[\bar{Y}_0=Y_0=0,...,\bar{Y}_n=\bar{y}_n,Y_n=y_n,\bar{Y}_{n+1}=\bar{y}_{n+1},...,\bar{Y}_{n+k}=\bar{y}_{n+k}|(Z_i=z_i)_{i<0}]}\\
&=\frac{\Pb_{m,\be_0,\be}[\bar{Y}_0=Y_0=0,...,\bar{Y}_n=\bar{y}_n,Y_n=y_n|(Z_i=z_i)_{i<0}]}{\Pb_{1,\be_0,\be_0}[\bar{Y}_0=Y_0=0,...,\bar{Y}_n=\bar{y}_n,Y_n=y_n|(Z_i=z_i)_{i<0}]}\\
&=\prod_{i=0}^{n-1}\frac{q_i(m,\be)(y,\bar{y})}{q_i(1,\be_0)(y,\bar{y})}
\end{align*}
Set $$Q_n(m,\be)=q_n(m,\be)(\bar{Y},Y), \F_n=\si\{(\bar{Y}_i)_{i\in\Zb},(Y_m)_{0\leq m\leq n}\}, M_n(m,\be)=\prod_{i=0}^{n-1}\frac{Q_i(m,\be)}{Q_i(1,\be_0)}.$$
We deduce that 
$$
\frac{d\Pb_{m,\be_0,\be}}{d\Pb_{1,\be_0,\be_0}}|_{\F_n}=M_n(m,\be),\ \frac{d\Pb_{m,\be_0,\be}}{d\Pb_{1,\be_0,\be_0}}|_{\F_{\tau}}=M_{\tau}(m,\be).
$$
We get the formula of the speed for $m-$excited random walk $Y$:
$$
v(m,\be)=\frac{\be}{d}\frac{\hat{\Eb}_{m,\be_0,\be}(N^m_{\tau})}{\hat{\Eb}_{m,\be_0,\be}(\tau)}=\frac{\be}{d}\frac{\hat{\Eb}_{m,\be_0,\be}(N^m_{\tau})}{\hat{\Eb}_{1,\be_0,\be_0}(\tau)},
$$
$$
\frac{\partial v}{\partial\be}(m,\be)=\frac{1}{d}\frac{\hat{\bar{\Eb}}_{1,\be_0}(N^m_{\tau}M_{\tau}(m,\be))}{\hat{\bar{\Eb}}_{1,\be_0}\tau}+\frac{\be}{d}\frac{\hat{\bar{\Eb}}_{1,\be_0}(N^m_{\tau}M_{\tau}(m,\be)V_{\tau}(m,\be))}{\hat{\bar{\Eb}}_{1,\be_0}\tau}.
$$
where 
$$
V_{\tau}(m,\be)=\frac{\frac{\partial}{\partial\be} M_{\tau}(m,\be)}{M_{\tau}(m,\be)}.
$$
Taking $m\to\infty$, by Lemma \ref{lemabasic} we get that $\frac{\partial v}{\partial\be}(m,\be)$ converges to $\frac{1}{d}$ uniformly in $\be\in[\be_0,1]$ when $m$ tends to infinity. This finishes the proof of Theorem.

\begin{center}
\huge{Acknowledgements}
\end{center}
I would like to thank my Ph.D. advisor  Pierre Mathieu for suggesting this problem. This research was supported by the
french ANR project MEMEMO2 2010 BLAN 0125.
\newpage
\bibliographystyle{plain}
\bibliography{D:/Bib/dl}
\end{document}